\documentclass[onefignum,onetabnum]{siamart251216}

\usepackage{mathdots}
\usepackage{amsfonts}
\usepackage{amssymb}
\usepackage{latexsym}
\usepackage{graphicx}
\usepackage{pgfplots}
\usepackage{epstopdf}
\usepackage{verbatim}
\usepackage{algorithm}
\usepackage{algorithmic}
\usepackage{hyperref}
\usepackage{amsopn,url}
\usepackage{amsmath}
\usepackage{amstext,amsfonts,bm,amssymb}
\usepackage{indentfirst}
\pgfplotsset{compat=1.18}

\def\subspace{\mathcal{Q}}
\def\Nystrom{Nystr{\"o}m}
\def\svdmethod{projected SVD}
\def\leadingsubspace{\mathcal{U}_1}
\def\trailingsubspace{\mathcal{U}_2}
\def\projector{\mathcal{P}}
\def\realR{\mathbb{R}}

\newcommand{\image}{\operatorname{span}}
\newcommand{\rank}{\operatorname{rank}}

\newcommand{\diag}{{\rm diag}}

\newtheorem{remark}{{\sc Remark}}[section]

\makeatletter
\renewcommand{\@seccntformat}[1]{\csname the#1\endcsname. \,}
\makeatother

\title{Finding accurate eigenvalues and eigenvectors of positive semi-definite matrices given a subspace}
\author{%
Yuji Nakatsukasa \and
Zheng Tang\thanks{Mathematical Institute, University of Oxford
  (\email{nakatsukasa@maths.ox.ac.uk}, \email{zheng.tang@maths.ox.ac.uk}). YN is supported by the EPSRC grant EP/Y030990/1.}%
}
\headers{Finding accurate eigenvalues given a subspace}{Y. Nakatsukasa and Z. Tang}

\begin{document}

\maketitle
\begin{abstract}
We revisit a classical problem in numerical linear algebra: given a $k$-dimensional subspace $\subspace$ with orthonormal basis $Q$ that approximates the leading eigenspace of an $n\times n$ positive semi-definite matrix $A$, the goal is to extract high-accuracy eigenvalues. The Rayleigh-Ritz (RR) method is the standard algorithm for the task, which has been shown to be optimal in several ways (when $A$ is symmetric, not necessarily positive semi-definite $A\succeq 0$). In this paper, we show that when $A\succeq 0$, alternative methods can outperform RR, including \Nystrom\ and one-sided projected SVD, while having the same computational complexity; that is, the main cost is in computing $AQ$, plus an $O(nk^2)$ term. In particular, we advocate the use of \Nystrom's method, quantitatively showing that the approximate leading eigenvalues have better accuracy than RR, and the improvement can be arbitrarily large. The significant accuracy improvement occurs for the first $i$ eigenvalues if $\lambda_i \gg \lambda_{k+1}$ and uniformly for all $k$ eigenvalues if $\lambda_k \gg \lambda_{k+1}$. This theoretical observation can be extended to the case where we extract $k$ leading eigenvalues from a larger-dimensional subspace.  A similar improvement is numerically observed for the purpose of approximating the leading eigenvectors. In contrast, when the target eigenvalues are the trailing ones, the situation is reversed, and the \Nystrom\ method performs poorly; we suggest a remedy based on a shift for this situation. 
\end{abstract}



\tableofcontents

\section{Introduction}
\label{sec:introduction}
The large-scale symmetric eigenvalue problem is ubiquitous and computationally challenging in numerical linear algebra. Solving such a problem plays an essential role in a wide range of applications not only in mathematics but also in various engineering fields~\cite{ambikasaran2015fast, axelsson2001finite, boffi2010finite, hu1985toeplitz, mastronardi1999computing, mena2018numerical}, and more recently in the machine learning community for kernel methods~\cite{belabbas2009spectral, gittens2016revisiting, williams2000using}.

In this paper, we mainly focus on the positive semi-definite case $A\succeq 0$. This is an important case for eigenvalue problems in machine learning and scientific computing applications, for example, computing the subset of the eigenpairs of the kernel matrix and the covariance matrix in kernel methods~\cite{abedsoltan2023toward, fowlkes2004spectral, meanti2020kernel, rudi2017falkon} and principal component analysis~\cite{bishop2006pattern}, respectively. When the matrix is too large for a dense algorithm (e.g., the QR algorithm) to be applicable, the standard approach is to turn to subspace methods. Here, one finds a low-dimensional subspace $\subspace$ (e.g., using a Krylov subspace or the subspace iteration method and their variants), and extracts eigenvalues and eigenvectors given $\subspace$. Rayleigh-Ritz (RR) has long been the de facto method for the task. In RR, one computes the approximate eigenvalues as those of $Q^\top AQ$, where $Q$ is an $n\times k$ matrix with orthonormal columns spanning the subspace $\subspace$. The complexity is that of computing $AQ$ and then $Q^\top(AQ)$ ($O(nk^2)$), in addition to finding the eigenvalues of $Q^\top AQ$ ($O(k^3)$), so $O(Nk+nk^2)$, where $N$ is the cost of performing a matrix-vector multiplication with $A$. 

When the matrix is positive semi-definite (PSD), additional techniques can be employed to address these problems by exploiting its structural information. Specifically, the PSD property allows us to adopt singular-value-based approaches, as it implies the equivalence between the eigenvalues and the singular values, and more generally, the eigendecomposition and the SVD. In particular, subspace iteration~\cite{saibaba2019randomized,persson2023randomized,saibaba2025randomized}, Krylov subspace methods~\cite[Sec.~12,~14]{parlett1998symmetric}~\cite[Sec.~5,~6]{saad2011numerical}, and randomized techniques~\cite{halko2011finding, martinsson2020randomized} can efficiently approximate the leading singular values of a general matrix, and hence the leading eigenspace of a PSD matrix. 

RR has long been known to be optimal in the sense of the min-max characterization and subspace projection~\cite[Ch.~11]{parlett1998symmetric}. Despite this, in this work we show that there are better techniques to approximate the leading eigenvalues that exploit the positive semi-definiteness, in particular, \Nystrom's method based on the eigenvalues of the \Nystrom\ approximation $A \langle Q \rangle= AQ (Q^\top A Q)^{\dagger} (AQ)^\top$ and \svdmethod\ based on the singular values of $AQ$. The resolution of the apparent inconsistency between the optimality and our result is as follows: the optimality of RR does not imply that the accuracy of the approximate eigenvalues cannot be improved, particularly when the PSD property is used. In this paper, we study the accuracy of the eigenvalue approximations obtained by three methods: RR (based on $\lambda_i(Q^\top AQ)$), \Nystrom\ (based on $\lambda_i(A\langle Q\rangle)$), and \svdmethod\ (based on $\sigma_i(AQ)$). We explain these methods in more detail in Section~\ref{subsec:methods}.

RR extracts eigenvectors from $\image(Q)$, whereas \Nystrom\ extracts them from $\image(AQ)$. \svdmethod\ allows both choices. The \Nystrom\ eigenvalue approximations can be interpreted as the Ritz values of $A$ on $\image(A^{1/2}Q)$. These two subspaces can be viewed as the results of one extra step of subspace iteration with $A$ and $A^{1/2}$, which can benefit the leading eigenpair computation.

We note that a related finding is reported in \cite{lazzarino2025matrix} for the purpose of extracting singular values from approximate (left and right) singular subspaces. This paper can be seen as a specialization to the case $A\succeq 0$, where the simplifications lead to the particularly attractive method of \Nystrom.

Section~\ref{sec:Background-and-Motivation-section} reviews the background information and the connections of these eigenpair extraction techniques. In Section~\ref{sec:Theoretical-analysis-section}, we compare the eigenvalue approximations obtained by the three methods and quantify the advantage of \Nystrom\ for leading eigenvalue targets in terms of the spectrum of $A$ and the accuracy of the given subspace. We also briefly discuss the extraction of trailing eigenvalues, for which \Nystrom\ turns out to be the worst of the three methods. We show that this can be remedied by working with a shifted matrix $\gamma I-A$ for $\gamma\geq\lambda_{\max}(A)$, although obtaining such an upper bound may require some extra cost. Finally, in Section~\ref{sec:Numerical-illustrations-section}, we present numerical experiments illustrating our findings.

We focus on RR, \svdmethod, and \Nystrom\ in this paper, because they are single-pass methods and have the same computational complexity $O(Nk+nk^2)$. This enables us to avoid addressing the trade-off between method accuracy and complexity.

\subsection{Contributions}
\label{subsec:contributions}
Our main contribution is to quantitatively compare the leading eigenvalue extraction errors of RR, \svdmethod\ and \Nystrom, given an approximate invariant subspace $\image(Q)$. Theorem~\ref{thm:main-theorem} establishes the error bounds for the leading $k$ eigenvalue approximations given a $k$-dimensional subspace. It quantitatively shows the accuracy improvement of \Nystrom\ over RR and \svdmethod\, and \svdmethod\ over RR are approximately $\lambda_i/\lambda_{k+1}$, $\lambda_i/(2\lambda_{k+1})$ and $2$, respectively. Corollary~\ref{corollary:oversampling-corollary} extends these bounds to the extraction of $k$ targets inside a larger $p$-dimensional prescribed subspace, which shows the similar accuracy improvement.

\subsection{Notation}
\label{subsec:notation}
We work with real matrices unless stated otherwise. Eigenvalues of symmetric matrices and singular values are denoted by $\lambda_i(\cdot)$ and $\sigma_i(\cdot)$, respectively, in nonincreasing order. For $A\succeq0$, we denote the eigendecomposition of it by $A=U\Lambda U^\top$, where $U=[u_1,\ldots,u_n]$ is orthogonal, $\Lambda=\diag(\lambda_1,\ldots,\lambda_n)$ and $\lambda_i=\lambda_i(A)$. Throughout the paper we denote the invariant subspaces spanned by the leading and trailing $k$ eigenvectors of $A$ by
\[
    \leadingsubspace:=\image([u_1,\ldots,u_k]),
    \qquad
    \trailingsubspace:=\image([u_{n-k+1},\ldots,u_n]).
\]
The notation $\|\cdot\|_2$ denotes the vector Euclidean norm or matrix spectral norm, $\|\cdot\|_F$ the Frobenius norm, and $B^\dagger$ the Moore-Penrose pseudoinverse. A matrix $Q$ is called orthonormal when $Q^\top Q=I$. For a subspace $\mathcal{S}$, $\projector_{\mathcal{S}}$ denotes its orthogonal projector, and $A\succeq B$ means that $A-B$ is positive semi-definite. For a test matrix $X$, we denote the \Nystrom\ approximation of a PSD matrix $A$ with respect to $X$ by $A\langle X\rangle:=AX(X^\top AX)^\dagger X^\top A$.

Let subspaces $\mathcal{S}$ and $\mathcal{T}$ have orthonormal bases $X$ and $Y$, with $d:=\dim(\mathcal{S})\leq\dim(\mathcal{T})$. Their principal angles $0\leq\theta_1\leq\cdots\leq\theta_d\leq\pi/2$ are defined by $\cos\theta_j=\sigma_j(X^\top Y)$ for $j=1,\ldots,d$ \cite{zhu2013angles}. We denote the principal angle matrix and the largest principal angle by $\Theta(\mathcal{S},\mathcal{T}):=\diag(\theta_1,\ldots,\theta_d)$ and $\theta_{\max}(\mathcal{S},\mathcal{T}):=\theta_d$. For $v\neq0$,
\[
    \sin\angle(v,\mathcal{S})=\frac{\|(I-\projector_{\mathcal{S}})v\|_2}{\|v\|_2},
    \qquad
    \sin\theta_{\max}(\mathcal{S},\mathcal{T})=\|(I-\projector_{\mathcal{T}})X\|_2.
\]

\section{Background}
\label{sec:Background-and-Motivation-section}
Subspace methods separate the construction of an approximate invariant subspace from the extraction of eigenpairs. In this paper, we study the extraction stage for a prescribed subspace $\image(Q)$. Section~\ref{subsec:methods} provides an introduction to the projection methods. After defining these methods, Section~\ref{subsec:Extraction-subspaces-and-connections-among-the-methods} discusses the connections regarding the eigenvector extraction subspaces and eigenvalue approximations of them. Section~\ref{subsec:computational-cost} compares their computational costs.

\subsection{Projection-based methods to extract eigenpairs given a subspace}
\label{subsec:methods}
Here, we provide an overview of three projection-based approaches, including RR, to accomplish the extraction task.
\subsubsection{Rayleigh-Ritz (RR)}
\label{subsubsec:method-introduction-RR}
RR, described in Algorithm~\ref{alg:Rayleigh-Ritz}, is classic and the most widely used method based on projecting the matrix and computing the eigenvalue decomposition of $Q^\top A Q$.

\begin{algorithm}
	\algsetup{linenosize=\small} \small
	\renewcommand{\algorithmicrequire}{\textbf{Input:}}
	\renewcommand{\algorithmicensure}{\textbf{Output:}}
	\caption{RR}
	\begin{algorithmic}[1]
		\label{alg:Rayleigh-Ritz}
		\REQUIRE Positive semi-definite $A \in \realR^{n \times n}$; orthonormal $Q \in \realR^{n \times k}$.
		\ENSURE Diagonal $\hat{\Lambda} \in \realR^{k \times k }$ and orthonormal $\hat{U} \in \realR^{n \times k}$ that approximate the desired eigenpairs, i.e., $A\hat{U} \approx \hat{U}\hat{\Lambda}$
		\STATE Compute $Q^\top A Q$.
		\STATE Compute the eigendecomposition $Q^\top A Q = \widetilde{\Omega} \hat{\Lambda} \widetilde{\Omega}^\top$, where $\tilde{\Omega}$ is orthogonal.
		\STATE $\hat{U} = Q \widetilde{\Omega}$.
	\end{algorithmic} 
\end{algorithm}

\subsubsection{One-sided projected SVD}
\label{subsubsec:method-introduction-SVD}
For a PSD matrix, the eigenvalues coincide with the singular values. In randomized numerical linear algebra, the one-sided projected SVD technique is a common way to approximate the singular pairs (so eigenpairs) of a matrix $A$ when a subspace $\subspace$ is given~\cite{gu2015subspace, halko2011finding, martinsson2020randomized}. We abbreviate it as \svdmethod. Given $Q$, the \svdmethod\ uses the singular values of $AQ$ as eigenvalue approximations
\begin{equation}
    \label{eq:SVD-eigenvalue-approximation}
    \hat{\lambda}_{i}^{\mathrm{SVD}} = \sigma_i(AQ).
\end{equation}
Let $r := \rank(AQ)$ and write the rank-$r$ SVD of $AQ$ as
\[
    AQ = W \Sigma V^\top,
\]
where $W \in \realR^{n \times r}, V \in \realR^{k \times r}$ with $W^\top W = V^\top V = I_r$. Let $w_i$ and $v_i$ denote the $i$-th column of $W$ and $V$, respectively. We have two choices of approximate eigenvectors
\begin{equation}
    \hat{u}_{i}^{\mathrm{SVD}_{Q}} = Q v_i, \qquad \hat{u}_{i}^{\mathrm{SVD}_{AQ}} = w_i, \qquad i=1, \ldots, r.
\end{equation}
The subscripts specify the eigenvector extraction subspaces, i.e., $\image(Q)$ for $\mathrm{SVD}_{Q}$ and $\image(AQ)$ for $\mathrm{SVD}_{AQ}$. Both variants use the same input orthonormal basis $Q$ and the same eigenvalue approximations \eqref{eq:SVD-eigenvalue-approximation}.

\begin{algorithm}
	\algsetup{linenosize=\small} \small
	\renewcommand{\algorithmicrequire}{\textbf{Input:}}
	\renewcommand{\algorithmicensure}{\textbf{Output:}}
	\caption{\svdmethod}
	\begin{algorithmic}[1]
		\label{alg:SVD}
		\REQUIRE PSD $A \in \realR^{n \times n}$; orthonormal $Q \in \realR^{n \times k}$.
		\ENSURE Nonzero extracted eigenpairs $(\hat{\Lambda}, \hat{U})$; the remaining $k-r$ eigenvalue approximations are zero if $r < k$.
		\STATE Compute $Y = AQ$.
		\STATE Compute the SVD $AQ = W \Sigma V^\top$ with $r = \rank(AQ)$.
		\STATE Set $\hat{\Lambda} = \Sigma$.
        \STATE Set $\hat{U} = Q V$ for $\mathrm{SVD}_{Q}$, or $\hat{U} = W$ for $\mathrm{SVD}_{AQ}$.
	\end{algorithmic} 
\end{algorithm}

\subsubsection{\Nystrom}
\label{subsubsec:method-introduction-Nystrom}
Originally introduced in machine learning~\cite{williams2000using}, and now a popular tool in randomized numerical linear algebra, \Nystrom\ is an attractive method to form a low-rank approximation for a PSD matrix and thus its eigenpairs \cite{halko2011finding, martinsson2020randomized, nakatsukasa2020sharp, tropp2023randomized}. It seeks eigenvalue approximations by
\[
    \hat{\lambda}_i = \lambda_i \left(A \langle Q \rangle\right)
\]
and the corresponding eigenvectors of $A\langle Q\rangle$ are eigenvector approximations. The pseudocode appears in Algorithm~\ref{alg:Nystrom}. Write
\[
    AQ=\widetilde Q\widetilde R,
    \qquad Q^\top AQ=L^\top L,
\]
where $\widetilde Q$ has orthonormal columns. Write $Z:=\widetilde R L^\dagger$ with
\[
    A\langle Q\rangle
    =\widetilde QZZ^\top\widetilde Q^\top.
\]
$Z$ is obtained by the linear solver described in Algorithm~\ref{alg:Nystrom} without explicitly forming $L^\dagger$. Computing the SVD
\[
    Z=\widetilde W\widetilde\Sigma\widetilde V^\top
\]
gives the \Nystrom\ eigenpair approximations
\[
    \widehat\Lambda=\widetilde\Sigma^2,
    \qquad \widehat U=\widetilde Q\widetilde W.
\]

In finite-precision computation, the implementation of \Nystrom\ may suffer from numerical instability. This is due to the ill-conditioning of the core matrix $Q^\top A Q$, which makes the computation of its pseudoinverse highly sensitive to round-off errors. To our knowledge, there are two main classes of proposed techniques to overcome this instability. One is to introduce an $\epsilon$-truncated Cholesky factorization for computing the pseudoinverse \cite{bucci2025numerical}, which is shown to be stable when $Q$ is a subsampling matrix that gives a local-maximum volume submatrix~\cite{damle2025estimating}, and observed to work well in practice more generally.  Another approach~\cite{tropp2017fixed} is to incorporate a small shift to ensure the smallest singular value of the core matrix is larger than the unit round-off. In our paper we choose the former approach, as it is more efficient and stability was not seen to cause issues (even without the $\epsilon$ truncation).

\begin{algorithm}
    \algsetup{linenosize=\small} \small
    \renewcommand{\algorithmicrequire}{\textbf{Input:}}
    \renewcommand{\algorithmicensure}{\textbf{Output:}}
    \caption{\Nystrom\ method}
    \begin{algorithmic}[1]
    \label{alg:Nystrom}
    \REQUIRE Positive semi-definite $A \in \realR^{n \times n}$; orthonormal $Q \in \realR^{n \times k}$; a truncation parameter $\epsilon > 0$
    \ENSURE Diagonal $\hat{\Lambda} \in \realR^{k \times k }$ and orthonormal $\hat{U} \in \realR^{n \times k}$ that approximate the desired eigenpairs, i.e., $A\hat{U} \approx \hat{U}\hat{\Lambda}$
    \STATE Compute $Y = AQ$ and its QR factorization $Y = \widetilde{Q}\widetilde{R}$.
    \STATE Compute $N = Q^\top Y$.
    \STATE Compute an $\epsilon$-truncated Cholesky factorization $N = L_{\epsilon}^\top L_{\epsilon}$ (in practice, standard Cholesky $N = L^\top L$ suffices).
    \STATE Solve the overdetermined linear systems $Z L_{\epsilon} = \widetilde{R}$ for $Z$ via back substitution.
    \STATE Compute the SVD of $Z = \widetilde{W} \widetilde{\Sigma} \widetilde{V}^\top$.
    \STATE $\hat{\Lambda} = \widetilde{\Sigma}^2$, $\hat{U} = \widetilde{Q}\widetilde{W}$.
    \end{algorithmic} 
\end{algorithm}

\subsection{Connections among RR, \svdmethod\ and \Nystrom}
\label{subsec:Extraction-subspaces-and-connections-among-the-methods}
Here we discuss the differences among three methods in terms of eigenvector and eigenvalue extraction. 

\subsubsection{Choice of eigenvector extraction subspace}
\label{subsubsec:Choice-of-eigenvector-extraction-subspace}
The same prescribed orthonormal basis $Q$ gives rise to two subspaces for eigenvector extraction. RR and $\mathrm{SVD}_{Q}$ find the approximate eigenvector from $\image(Q)$, whereas \Nystrom\ and $\mathrm{SVD}_{AQ}$ extract from $\image(AQ)$. We now explain when $\image(Q)$ methods and $\image(AQ)$ methods would be preferable over one another.

Intuitively, for any vector $v \in \realR^n$, we expand it and $Av$ in the eigenbasis of $A$
\[
    v = \sum_{j=1}^{n}c_j u_j, \qquad Av = \sum_{j=1}^{n}\lambda_j c_j u_j.
\]
Provided $Av \neq 0$, $Av$ can reduce the contribution of the unwanted trailing eigendirections for a vector close to $\leadingsubspace$, while amplifying the contribution of the unwanted leading eigendirections for a vector close to $\trailingsubspace$. In other words, the one extra step of subspace iteration $Q \mapsto AQ$ benefits the leading eigenvector extractions but harms trailing eigenvector extractions. The following proposition quantifies this using the largest principal angle.

\begin{proposition}
    Let $A$ be PSD and $1\leq k<n$. Suppose that $Q^\top Q=I_k$ and $\rank(AQ)=k$. Then,
    \[
        \theta_{\max} \left(\image(AQ), \leadingsubspace\right) \leq \theta_{\max} \left(\image(Q), \leadingsubspace\right),
    \]
    \[
        \theta_{\max} \left(\image(AQ), \trailingsubspace\right) \geq \theta_{\max} \left(\image(Q), \trailingsubspace\right).
    \]
\end{proposition}
\begin{proof}
    Let subspace $\mathcal{T}$ denote either $\leadingsubspace$ or $\trailingsubspace$. For any non-zero vector $v \in \image(Q)$, decompose $v = v_1 + v_2$ with $v_1 = \projector_{\mathcal{T}} v$ and $v_2 = (I - \projector_{\mathcal{T}}) v$. Then,
    \[
    \begin{aligned}
        \sin^2 \angle \left(Av, \mathcal{T}\right) - \sin^2 \angle \left(v, \mathcal{T}\right)
        & =
        \frac{\|(I - \projector_{\mathcal{T}}) Av\|_2^2}{\|Av\|_2^2} - \frac{\|(I - \projector_{\mathcal{T}}) v\|_2^2}{\|v\|_2^2} \\
        & = \frac{\|Av_2\|_2^2}{\|Av_1\|_2^2 + \|Av_2\|_2^2} - \frac{\|v_2\|_2^2}{\|v_1\|_2^2 + \|v_2\|_2^2} \\
        & = \frac{\|Av_2\|_2^2 \|v_1\|_2^2 - \|Av_1\|_2^2\|v_2\|_2^2}{\left(\|Av_1\|_2^2 + \|Av_2\|_2^2\right) \left(\|v_1\|_2^2 + \|v_2\|_2^2\right)}.
    \end{aligned}
    \]
    The denominators are positive because the rank assumption implies $Av \neq 0$. For $\mathcal{T} = \leadingsubspace$, we have $\|Av_2\|_2^2 \|v_1\|_2^2 - \|Av_1\|_2^2\|v_2\|_2^2 \leq 0$ since $\lambda_{k}^2 \|v_1\|_2^2 \leq \|A v_1\|_2^2$ and $\|A v_2\|_2^2 \leq \lambda_{k+1}^2 \|v_2\|_2^2$. Similarly, for $\mathcal{T} = \trailingsubspace$, the bounds $\lambda_{n-k+1}^2 \|v_1\|_2^2 \geq \|A v_1\|_2^2$ and $\|A v_2\|_2^2 \geq \lambda_{n-k}^2 \|v_2\|_2^2$ give  $\|Av_2\|_2^2 \|v_1\|_2^2 - \|Av_1\|_2^2\|v_2\|_2^2 \geq 0$. Finally, 
    \[
        \sin \theta_{\max} \left(\image(AQ), \mathcal{T}\right) = \sup_{v \in \image(Q), Av \neq 0} \sin \angle\left(Av, \mathcal{T}\right), \qquad \mathcal{T} \in \{\leadingsubspace, \trailingsubspace\},
    \]
    whereas replacing $Av$ by $v$ gives the corresponding expression for $\image(Q)$. Taking suprema proves the two inequalities.
\end{proof}
This comparison concerns the subspaces for extraction. It allows equality and does not imply an ordering of the errors of individual extracted eigenvectors.

\subsubsection{Algebraic connections among the methods}
\label{subsubsec:algebraic-connections-among-the-methods}
The eigenvalue approximations have complementary algebraic interpretations. Define the square root of a PSD matrix $A$ via its eigendecomposition $A^{1/2} := U \Lambda^{1/2} U^\top$. For $i = 1, \ldots, k$, we have 
\[
    \hat{\lambda}_{i}^{(\mathrm{RR})}(A) = \lambda_{i}(Q^\top A Q) = \left(\sigma_{i}(A^{1/2}Q)\right)^2 = \left(\hat{\lambda}_{i}^{(\mathrm{SVD})}(A^{1/2})\right)^2,
\]
\[
    \hat{\lambda}_{i}^{(\mathrm{SVD})}(A) = \sigma_{i}(AQ) = \sqrt{\lambda_{i}(Q^\top A^2 Q)} = \sqrt{\hat{\lambda}_{i}^{(\mathrm{RR})}(A^2)},
\]
where $i = 1, \ldots, \dim(\image(Q))$. Thus, from the projected SVD perspective, the RR eigenvalue estimates are equal to those of (the square of) \svdmethod\ applied to $A^{1/2}$ with the same $Q$; from the RR perspective, the \svdmethod\ approximation to the eigenvalues of $A$ can be viewed as the square root of RR eigenvalue approximations for $A^2$. Both identities use the same prescribed $Q$.

The second connection is that the eigenvalue approximations obtained by RR and \Nystrom\ can both be interpreted as Ritz values of $A$, but with respect to different subspaces. For RR, we have
\begin{equation}
    \label{eq:RR-eigenvalue-compressed-operator-representation}
    \hat{\lambda}_{i}^{(\mathrm{RR})}(A) = \lambda_i(Q^\top A Q) = \lambda_i(\projector_{\image(Q)} A \projector_{\image(Q)}), \qquad i=1,\ldots, k.
\end{equation}
For \Nystrom\ \cite{gittens2016revisiting} and $i=1,\ldots, k$, we have
\begin{equation}
    \label{eq:Nystrom-eigenvalue-compressed-operator-representation}
    \begin{aligned}
        \hat{\lambda}_{i}^{(\mathrm{Nys})}(A) 
        = 
        \lambda_i\left((AQ) (Q^\top A Q)^{\dagger} (AQ)^\top\right) 
        & = 
        \lambda_i\left(A^{1/2} \projector_{\image(A^{1/2} Q)} A^{1/2}\right) \\
        & = 
        \lambda_i\left(\projector_{\image(A^{1/2} Q)} A \projector_{\image(A^{1/2} Q)}\right),
    \end{aligned}
\end{equation}
where the second equality is due to the fact that $B^\top B$ and $BB^\top$ have the same non-zero eigenvalues. Consequently, the RR eigenvalue approximations are the Ritz values of $A$ on $\image(Q)$, whereas the \Nystrom\ eigenvalue approximations are the Ritz values of $A$ on $\image(A^{1/2}Q)$. The latter subspace benefits from one extra step of the subspace iteration for $A^{1/2}$, so it is at least as accurate as $\image(Q)$ in approximating $\leadingsubspace$ in the largest principal angle sense.

Note that the Ritz value representation of \Nystrom\ involves $\image(A^{1/2} Q)$, concerning the eigenvalues, whereas its approximate eigenvectors associated with nonzero eigenvalues lie in $\image(AQ)$. These two subspaces have different roles. Neither $A^{1/2}$ nor $A^2$ needs to be explicitly formed in the algorithms.

The third connection is through the projection representations on $\image(Q_F)$, where $Q_F:=[Q, Q_{\perp}]$. In these coordinates, the RR eigenvalues are the leading eigenvalues of
\[
    [Q\ Q_\perp]^\top (QQ^\top A QQ^\top)[Q\ Q_\perp]
=
\begin{bmatrix}
B & 0\\
0 & 0
\end{bmatrix},
\]
where $B:=Q^\top A Q$. In contrast, the \Nystrom\ eigenvalues are leading eigenvalues of
\begin{equation}
    \label{eq:Nystrom-block-representation}
    [Q\ Q_\perp]^\top A\langle Q\rangle
    [Q\ Q_\perp]
    =
    \begin{bmatrix}
    B & C^\top\\
    C & CB^\dagger C^\top
    \end{bmatrix},
\end{equation}
where $C = Q_{\perp}^\top A Q$ and $CB^\dagger C^\top \succeq 0$ \cite{gittens2016revisiting}. Since 
\[
    [Q\ Q_\perp]^\top A [Q\ Q_\perp]
    =
    \begin{bmatrix}
    B & C^\top\\
    C & D
    \end{bmatrix}
\]
and $D-CB^\dagger C^\top\succeq0$ \cite[proof of Lemma~5.2]{tropp2023randomized}, \Nystrom\ preserves the leading projection $B$, the off-diagonal coupling $C$ and $CB^\dagger C^\top$, while RR only preserves $B$. Moreover, the Ritz value representation of \Nystrom\ is a direct consequence of its $Q_F$-coordinate representation. When $B\succ0$, since
\[
    \begin{bmatrix}
    B & C^\top\\
    C & CB^\dagger C^\top
    \end{bmatrix}
    =
    \begin{bmatrix}
    B^{1/2} \\
    C B^{-1/2}
    \end{bmatrix}
    \begin{bmatrix}
    B^{1/2} & B^{-1/2}C^\top
    \end{bmatrix},
\]
the nonzero eigenvalues of $A\langle Q\rangle$ are those of
\[
    B+B^{-1/2}C^\top CB^{-1/2}
    =
    (A^{1/2}QB^{-1/2})^\top A(A^{1/2}QB^{-1/2}).
\]
Since $A^{1/2}QB^{-1/2}$ is orthonormal, this is the same as the Ritz representation
\eqref{eq:Nystrom-eigenvalue-compressed-operator-representation}.

\begin{remark}[Basis invariance]
\label{remark:basis-invariance-property}
The RR, \svdmethod\ and \Nystrom\ eigenvalues are invariant under an orthogonal change of basis $Q\mapsto QV$. For RR and \svdmethod, this follows from the orthogonal similarity of $Q^\top A Q$ and $Q^\top A^2 Q$. For \Nystrom, the identity $A\langle QM\rangle=A\langle Q\rangle$ holds for any nonsingular $M\in\realR^{k\times k}$ \cite{chen2025randomly}.
\end{remark}

Since all three methods already form $AQ$, using $\image(AQ)$ for eigenvector extraction exploits a matrix product already available in the computation. The following subsection compares the arithmetic costs.

\subsection{Computational cost}
\label{subsec:computational-cost}
Computational cost is an important aspect of performance. Here, we measure the cost in terms of floating-point operations. Table~\ref{table:method-complexity-table} summarizes the complexity of these three methods, including the constants. All methods are $O(Nk+nk^2)$, thereby allowing us to largely avoid the trade-off between method accuracy and complexity. The cost of constructing $Q$ is not included.

\begin{table}[ht]
    \centering
    \caption{The complexity of different methods, where $k = \dim(\subspace)$ and $N$ is the cost of performing a matrix-vector multiplication with $A$.
    }
    \label{table:method-complexity-table}
    \begin{tabular}{c|c}
        \hline
        & Method complexity \\
        \hline
        Rayleigh-Ritz & $Nk + 2nk^2$ \\
        \svdmethod & $Nk + 4nk^2$ \\
        \Nystrom & $Nk+ 4nk^2$ \\
        \hline
    \end{tabular}
\end{table}

While the cost of RR is slightly lower than \Nystrom, we will see that the latter can give significantly better approximations to the eigenpairs, making it well worth the extra cost.

\section{Theoretical analysis}
\label{sec:Theoretical-analysis-section}
In Section~\ref{subsec:eigenvalue-error-ordering-subsection}, we first recall a known ordering of RR, \svdmethod\ and \Nystrom\ eigenvalues and then discuss a consequent error ordering among them when $\image(Q)$ approximates either $\leadingsubspace$ or $\trailingsubspace$. The error ordering is independent of subspace accuracy. In Section~\ref{subsec:Quantitative-comparison-of-leading-eigenvalue-errors}, we address the separate question of quantitative error bounds when $\image(Q)$ approximates $\leadingsubspace$. The error bounds further imply that the improvement in accuracy can be significant by switching from \svdmethod\ and RR to \Nystrom\, when the spectrum of $A$ satisfies certain conditions; see Section~\ref{subsec:Quantitative-comparison-of-leading-eigenvalue-errors} and the discussion after Corollary~\ref{corollary:oversampling-corollary} for details.

\subsection{Ordering of the eigenvalue approximations and error comparison}
\label{subsec:eigenvalue-error-ordering-subsection}
We first provide a proposition about the ordering of eigenvalue approximations $\hat{\lambda}_{i}^{(\mathrm{RR})}$, $\hat{\lambda}_{i}^{(\mathrm{SVD})}$ and $\hat{\lambda}_{i}^{(\mathrm{Nys})}$. Persson, Meyer and Musco~\cite[proof of Theorem~3.7]{persson2025algorithm} showed that the \Nystrom\ eigenvalue is at least as large as the corresponding \svdmethod\ one. Combining their result with the standard inequality $\sigma_i(AQ)\geq\lambda_i(Q^\top A Q)$ gives the full ordering.
\begin{proposition}
    \label{proposition:eigenvalue-approximation-ordering}
    Let $A\in\realR^{n\times n}$ be PSD and let $Q\in\realR^{n\times k}$ satisfy $Q^\top Q=I_k$, where $1\leq k\leq n$. Then,
    \begin{equation}
        \label{eq:eigenvalue-approximation-ordering}
        \lambda_i(A\langle Q\rangle)\geq\sigma_i(AQ)\geq\lambda_i(Q^\top AQ),\qquad i=1,\ldots,k.
    \end{equation}
    Neither approximation of an invariant subspace nor $\operatorname{rank}(AQ)=k$ is required.
\end{proposition}

\begin{proof}
From the proof of Theorem~3.7 of~\cite{persson2025algorithm}, we have
\[
    \lambda_i(A\langle Q\rangle) \geq \sigma_i(Q^\top A\langle Q\rangle) = \sigma_i(Q^\top A).
\]
The remaining inequality follows from $Q^\top A Q \succeq 0$ and $\lambda_i(Q^\top A Q) = \sigma_i(Q^\top A Q) \leq \sigma_i(AQ)$.
\end{proof}

The eigenvalue ordering \eqref{eq:eigenvalue-approximation-ordering} is universal since it requires only that $Q$ be orthonormal. When $\image(Q)$ is constructed to approximate $\leadingsubspace$, the eigenvalue approximations are matched with $\lambda_1(A), \ldots,\lambda_k(A)$. In contrast, when $\image(Q)$ approximates $\trailingsubspace$, they are matched, in the same nonincreasing order, with $\lambda_{n-k+1}(A), \ldots,\lambda_n(A)$. These two choices lead to two eigenvalue error orderings, but the accuracy of $\image(Q)$ to either target invariant subspace is not a hypothesis of the following inequalities themselves.
\begin{proposition}
    \label{proposition:Nystom-vs-SVD-vs-RR-eigenvalue-error-ordering}
    Under the assumptions of Proposition~\ref{proposition:eigenvalue-approximation-ordering}, the following statements hold for $i=1, \ldots,k$.
    
    \textup{(i) Leading targets:} If the eigenvalue approximations are matched with $\lambda_i(A)$, then
    \begin{equation}
        \label{eq:leading-eigenvalue-error-ordering}
        0 
        \leq 
        \left|\lambda_i(A)-\hat\lambda_i^{(\mathrm{Nys})}\right|
        \leq
        \left|\lambda_i(A)-\hat\lambda_i^{(\mathrm{SVD})}\right|
        \leq
        \left|\lambda_i(A)-\hat\lambda_i^{(\mathrm{RR})}\right|.
    \end{equation}
    \textup{(ii) Trailing targets:} If the approximations are matched with $\lambda_{n-k+i}(A)$, then instead
    \begin{equation}
        \label{eq:trailing-eigenvalue-error-ordering}
        0
        \leq
        \left|\hat\lambda_i^{(\mathrm{RR})}-\lambda_{n-k+i}(A)\right|
        \leq
        \left|\hat\lambda_i^{(\mathrm{SVD})}-\lambda_{n-k+i}(A)\right|
        \leq
        \left|\hat\lambda_i^{(\mathrm{Nys})}-\lambda_{n-k+i}(A)\right|.
    \end{equation}
\end{proposition}
\begin{proof}
The fact $0\preceq A\langle Q\rangle\preceq A$ \cite{tropp2017fixed, persson2023randomized} implies $\hat\lambda_i^{\mathrm{Nys}}\leq\lambda_i(A)$. The interlacing inequality gives $\lambda_i(Q^\top AQ)\geq\lambda_{n-k+i}(A)$ \cite[Section~4.3.2]{saad2011numerical}. Then, 
\begin{equation}
    \lambda_{n-k+i}(A)
    \leq
    \hat\lambda_i^{\text{(RR)}}
    \leq
    \hat\lambda_i^{\text{(SVD)}}
    \leq
    \hat\lambda_i^{\text{(Nys)}}
    \leq
    \lambda_i(A),
    \label{eq:Nystrom-vs-SVD-vs-RR}
\end{equation}
which completes the proof.
\end{proof}
Note that neither subspace target guarantees small errors for an arbitrary $Q$. Subspace accuracy becomes important when bounding the magnitude of those errors, as in Section~\ref{subsec:Quantitative-comparison-of-leading-eigenvalue-errors}.

\paragraph{A shift remedy for trailing eigenpairs}
\label{par:shifted-trailing}
Although RR gives the most accurate trailing eigenvalue errors among the three methods applied directly to $A$ with the same $Q$, the leading eigenvalue advantage of \Nystrom\ can be exploited by reversing the spectrum. Let
\[
    B = \gamma I - A, \qquad \gamma \geq \lambda_{1}(A),
\]
where $\gamma$ is an available upper bound, $B \succeq 0$ and its leading $k$ eigenvectors span the trailing invariant subspace of $A$. Define the recovered eigenvalue approximations, ordered nonincreasingly 
\[
    \hat{\lambda}_{i}^{\mathrm{Nys, shift}} := \gamma - \lambda_{k-i+1} (B \langle Q \rangle), \qquad i=1, \ldots, k.
\]
Since $Q^\top B Q = \gamma I - Q^\top A Q$, in exact arithmetic, applying RR to $B$ and transforming back gives exactly the same eigenvalue approximation as applying RR directly to $A$. Applying the leading eigenvalue error ordering \eqref{eq:leading-eigenvalue-error-ordering} to $B$, we have
\[
    \gamma - \lambda_{n-k+i}(A) \geq \lambda_{k-i+1}(B \langle Q \rangle) \geq \lambda_{k-i+1}(Q^\top B Q) = \gamma - \lambda_{i}(Q^\top A Q), \quad i=1, \ldots, k,
\]
which implies
\[
    \lambda_{n-k+i}(A) \leq \hat{\lambda}_{i}^{\text{Nys, shift}} \leq \hat{\lambda}_{i}^{\mathrm{RR}}, \quad i=1, \ldots, k.
\]
Thus, in exact arithmetic, the shifted \Nystrom\ is at least as accurate as direct RR for the trailing eigenvalue targets. Once $AQ$ is available, $BQ = \gamma Q - AQ$ requires no further multiplication by $A$, although obtaining a reliable $\gamma$ may require additional cost. In finite precision, subtraction in $\gamma-\lambda_{k-i+1}(B\langle Q\rangle)$ can limit the relative accuracy of small recovered eigenvalues. Section~\ref{sec:Numerical-illustrations-section} numerically illustrates the shifted extraction.

\subsection{Quantitative comparison of leading eigenvalue errors}
\label{subsec:Quantitative-comparison-of-leading-eigenvalue-errors}
The discussion in Section~\ref{subsec:eigenvalue-error-ordering-subsection} orders the eigenvalue extraction errors without quantifying them. We now quantitatively estimate the eigenvalue errors themselves given a subspace $\image(Q)$ close to the leading invariant subspace $\leadingsubspace$. The results bound the eigenvalue approximation errors in terms of the subspace accuracy and the spectrum of $A$, thereby implying the improvement ratio when we switch methods. We first consider a $k$-dimensional prescribed search subspace for extracting exactly $k$ leading eigenvalues in Theorem~\ref{thm:main-theorem}. A larger dimensional search subspace with the same $k$ target eigenvalues is considered in Corollary~\ref{corollary:oversampling-corollary}. 

The eigenvalue approximations from RR, \svdmethod\ and \Nystrom\ are invariant under orthogonal change of basis on the right; see Remark~\ref{remark:basis-invariance-property}. Indeed, for any orthogonal $V \in \realR^{k \times k}$ with $V^\top V = I_k$,
\[
    \lambda_i((QV)^\top A (QV)) = \lambda_i(Q^T AQ), \quad \sigma_i(AQV) = \sigma_i(AQ), \quad \lambda_i(A\langle QV\rangle) = \lambda_i(A\langle Q\rangle).
\]
Note that this is a change of basis within $\image(Q)$, not a rotation of the fixed eigenbasis. Let
\[
    A = U \Lambda U^\top, \qquad U=[U_1\ U_{2}], \qquad \Lambda=\diag(\Lambda_1,\Lambda_2),
\]
where $U_1=[u_1,\ldots,u_k]$ and $\Lambda_1=\diag(\lambda_1,\ldots,\lambda_k)$. Let $Q$ be the orthonormal basis of the given subspace $\subspace$. To quantify the accuracy of $\subspace$ as an approximation to $\leadingsubspace$, we express an orthonormal basis of $\subspace$ in a structural form relative to $U_1$ and $U_2$ by
\begin{equation}
    \label{eq:Q-in-U-coordinate}
    U^\top Q V
    =
    \begin{bmatrix}
    I_k+\epsilon^2M\\
    \epsilon N
    \end{bmatrix},
    \qquad 0\leq\epsilon<1,\qquad \|M\|_2,\|N\|_2\leq1.
\end{equation}
The lower block in \eqref{eq:Q-in-U-coordinate} controls the accuracy of $\subspace=\image(Q)$ as an approximation to $\leadingsubspace=\image(U_1)$, since
\begin{equation}
    \label{eq:sine-angle-subspace-discussion}
    \|\sin\Theta(\leadingsubspace,\subspace)\|_2
    =\|U_2^\top Q\|_2
    =\|U_2^\top Q V\|_2
    =\epsilon\|N\|_2\leq\epsilon.
\end{equation}
The leading block describes the alignment of the basis $QV$ with $U_1$, where $\|M\|_2\leq1$ ensures $\|U_1^\top QV-I_k\|_2\leq\epsilon^2$. Our structural assumption is a natural representation of subspace accuracy, rather than an additional restriction on $Q$, since such a second-order correction can always be achieved by an orthogonal change of basis within $\subspace$. In particular, suppose conversely $\|U_2^\top Q\|_2 \leq \epsilon < 1$, the polar decomposition allows us to choose $U_1^\top Q = H R$ with $R$ orthogonal and $H\succ0$. Let $V=R^\top$ and $T:=U_2^\top Q V$, we have $H=(I_k-T^\top T)^{1/2}$ and
\[
    \|H-I_k\|_2 \leq 1-\sqrt{1-\epsilon^2}=\frac{\epsilon^2}{1+\sqrt{1-\epsilon^2}}\leq\epsilon^2.
\]
For $\epsilon>0$, choose $M=(H-I_k)/\epsilon^2$ and $N=T/\epsilon$. Thus, \eqref{eq:Q-in-U-coordinate} is equivalent, up to a choice of basis, to the principal angle sine bound \eqref{eq:sine-angle-subspace-discussion}. The former gives a clearer structure of the behavior of $Q$ in $U$-coordinates than the latter. The structural condition \eqref{eq:Q-in-U-coordinate} also automatically guarantees that $U_1^\top Q$ is nonsingular or equivalently, $\theta_{\max} < \pi /2$. However, it should be noted that this provides no quantitative small principal angle guarantee.

In exact arithmetic, if $\lambda_k>\lambda_{k+1}$, subspace iteration whose initialization matrix $\Omega$ admits nonsingularity in $U_1^\top \Omega$ can achieve any prescribed sine principal angle tolerance $0<\epsilon<1$ by sufficiently many iterations. The discussion following Corollary~\ref{corollary:oversampling-corollary} provides more details. We now bound the extraction errors through the Ritz value representations and quantify the error improvement. We present the derivation of the error bounds for the three methods in a consistent manner, illustrating how different spectral factors arise from the choice of the orthonormal basis and the trial subspace. The resulting RR bound is classical \cite{knyazev1985sharp}.

\begin{theorem}
\label{thm:main-theorem}
Let $A\in\realR^{n\times n}$ be PSD, with
\[
    A=U\Lambda U^\top,\qquad
    U=[U_1\ U_2],\qquad
    \Lambda=\diag(\Lambda_1,\Lambda_2),
\]
where $U$ is orthogonal, $\Lambda_1=\diag(\lambda_1,\ldots,\lambda_k)$, $\Lambda_2=\diag(\lambda_{k+1},\ldots,\lambda_n)$, $\lambda_1\geq\cdots\geq\lambda_n\geq0$, and $1\leq k<n$. Assume $\lambda_k>0$ and let $Q\in\realR^{n\times k}$ satisfies $Q^\top Q=I_k$. Suppose there exists an orthogonal $V\in\realR^{k \times k}$ such that
\begin{equation}
    \label{eq:Q-in-U-coordinate-thm}
    QV=U_1+\epsilon^2U_1M+\epsilon U_{1,\perp}N, \qquad 0\leq\epsilon<1,\qquad \|M\|_2,\|N\|_2\leq1,
\end{equation}
where $M\in\realR^{k\times k}$ and $N\in\realR^{(n-k)\times k}$. Then, for $i=1,\ldots,k$,
\begin{equation}
    \label{eq:error-rr}
    \left|\lambda_i-\lambda_i(Q^\top AQ)\right|
    \leq
    (\lambda_i-\lambda_n)\epsilon^2
    \leq
    \lambda_i\epsilon^2,
\end{equation}
\begin{equation}
    \label{eq:error-svd}
    \left|\lambda_i-\sigma_i(AQ)\right|
    \leq
    \lambda_i-\sqrt{(1-\epsilon^2)\lambda_i^2+\epsilon^2\lambda_n^2}
    \leq
    \frac{\lambda_i\epsilon^2}{2}+\frac{\lambda_i\epsilon^4}{2(1+\sqrt{1-\epsilon^2})^2}.
\end{equation}
If, in addition,
\begin{equation}
    \label{eq:abs-local-condition}
    \epsilon^2\leq\frac{\lambda_{k+1}}{\lambda_1},
\end{equation}
then the \Nystrom\ error satisfies
\begin{equation}
    \label{eq:error-nys-sharp}
    \left|\lambda_i-\lambda_i(A\langle Q\rangle)\right|
    \leq\lambda_{k+1}\epsilon^2,
    \qquad i=1,\ldots,k.
\end{equation}
Without requiring~\eqref{eq:abs-local-condition}, the RR and projected-SVD bounds above remain valid, and \Nystrom\ satisfies
\begin{equation}
    \label{eq:error-nys}
    \left|\lambda_i-\lambda_i(A\langle Q\rangle)\right|
    \leq\frac{\lambda_{k+1}\epsilon^2}
    {1-(1-\lambda_{k+1}/\lambda_i)\epsilon^2}
    =\lambda_{k+1}\epsilon^2+
    \frac{\lambda_{k+1}(1-\lambda_{k+1}/\lambda_i)\epsilon^4}
    {1-(1-\lambda_{k+1}/\lambda_i)\epsilon^2}.
\end{equation}
\end{theorem}

\begin{proof}
By the basis invariance in Remark~\ref{remark:basis-invariance-property}, we may assume without loss of generality that $Q$ itself satisfies the structural condition \eqref{eq:Q-in-U-coordinate-thm}. Since $Q^\top Q = I_k$ and $U^\top U = I_n$, 
\[
    (I_k+\epsilon^2M)^\top(I_k+\epsilon^2M)=I_k-\epsilon^2N^\top N\succeq(1-\epsilon^2)I_k\succ0,
\]
and $U_1^\top Q=I_k+\epsilon^2M$ is nonsingular. Define $Z:=(U_2^\top Q)(U_1^\top Q)^{-1}=\epsilon N(I_k+\epsilon^2M)^{-1}$ and $\tau:=\|Z\|_2$. By definition we have
\begin{equation}
    \label{eq:abs-angle-transfer}
    \frac{\tau^2}{1+\tau^2}
    =\epsilon^2\|N\|_2^2\leq\epsilon^2,
    \qquad (1-\epsilon^2)\tau^2\leq\epsilon^2.
\end{equation}
Since $\lambda_k > 0$,
\[
    Q=U\begin{bmatrix}I_k\\Z\end{bmatrix}(I_k+\epsilon^2M),
    \qquad
    A^{1/2}Q=
    U\begin{bmatrix}\Lambda_1^{1/2}\\\Lambda_2^{1/2}Z\end{bmatrix}
    (I_k+\epsilon^2M).
\]
and further
\[
    \image(Q)
    =
    \image
    \left(
    U
    \begin{bmatrix}
    I_k\\
    Z
    \end{bmatrix}
    \right),
    \qquad
    \image(A^{1/2}Q)
    =
    \image
    \left(
    U
    \begin{bmatrix}
    \Lambda_1^{1/2}\\
    \Lambda_2^{1/2}Z
    \end{bmatrix}
    \right).
\]
This implies that the expression of RR and \Nystrom\ eigenvalues as the Ritz values of $A$ on subspaces $\image(Q)$ and $\image(A^{1/2}Q)$ can be expressed using the non-orthonormal basis $[I; Z]$ and $[\Lambda_1^{1/2}; \Lambda_2^{1/2} Z]$ in $U$-coordinates. Define subspaces
\[
    \mathcal{T}(Z)
    :=
    \image
    \left(
    \begin{bmatrix}
    I_k \\
    Z
    \end{bmatrix}
    \right),
    \qquad
    \mathcal{V}
    :=
    \image
    \left(
    \begin{bmatrix}
    \Lambda_1^{1/2} \\
    \Lambda_2^{1/2}Z
    \end{bmatrix}
    \right).
\]
In $U$-coordinates, the Ritz value representations give
\begin{equation}
    \label{eq:ritz-representation-in-U-coordinate}
    \hat{\lambda}_i^{(\mathrm{RR})}=\rho_i(\Lambda;\mathcal{T}(Z)),
    \qquad
    \left(\hat{\lambda}_i^{(\mathrm{SVD})}\right)^2=\rho_i(\Lambda^2;\mathcal{T}(Z)),
    \qquad
    \hat{\lambda}_i^{(\mathrm{Nys})}=\rho_i(\Lambda;\mathcal{V}),
\end{equation}
where $i=1,\ldots,k$ and $\rho_i(\Lambda;\mathcal{S}):=\lambda_i(\projector_{\mathcal{S}} \Lambda \projector_{\mathcal{S}})$ denotes the $i$-th largest Ritz value of $\Lambda$ on subspace $\mathcal{S}$.

For RR and \svdmethod, fix $i\leq k$ and define $\mathcal{E}_i:=\image([e_1,\ldots,e_i])\subseteq\realR^k$ and $\mathcal{L}_i := \{[I_k; Z]x: x \in \mathcal{E}_i\}$, where $e_1,\ldots,e_k$ are the standard basis vectors of $\realR^k$. Since $x\mapsto[x;Zx]$ is injective, $\dim(\mathcal{L}_i) = \dim(\mathcal{E}_i) = i$. For $q=1,2$, we have
\begin{equation}
    \label{eq:abs-maxmin}
    \rho_i(\Lambda^q;\mathcal{T}(Z))
    =
    \max_{\substack{\mathcal{L}\subseteq\mathcal{T}(Z)\\\dim\mathcal{L}=i}}\ \min_{0\neq y\in\mathcal{L}}\frac{y^\top\Lambda^qy}{y^\top y}
    \geq
    \min_{\substack{y = \left[\begin{smallmatrix}I_k\\Z\end{smallmatrix}\right]x,\\ x\in\mathcal{E}_i,\, \|x\|_2=1}}\frac{x^\top\Lambda_1^qx+(Zx)^\top\Lambda_2^q(Zx)}{1+\|Zx\|_2^2}.
\end{equation}
Let
\[
    \eta_i:=\max_{\substack{x\in\mathcal E_i\\\|x\|_2=1}}\|Zx\|_2 \leq \tau,
    \qquad
    \mu:=\frac{\lambda_i^q+\lambda_n^q\eta_i^2}{1+\eta_i^2}.
\]
For every unit vector $x\in\mathcal E_i$, since $\lambda_n^q\leq\mu\leq\lambda_i^q$,
\[
    \begin{aligned}
    x^\top\Lambda_1^qx+(Zx)^\top\Lambda_2^q(Zx)
       -\mu(1+\|Zx\|_2^2)
    &\quad\geq\lambda_i^q-\mu
       -(\mu-\lambda_n^q)\|Zx\|_2^2\\
    &\quad\geq\lambda_i^q-\mu
       -(\mu-\lambda_n^q)\eta_i^2=0.
    \end{aligned}
\]
Then, we can bound \eqref{eq:abs-maxmin} below by $\mu$ and obtain
\[
    \begin{aligned}
    \rho_i(\Lambda^q;\mathcal{T}(Z))
    \geq
    \frac{\lambda_i^q+\lambda_n^q\eta_i^2}{1+\eta_i^2}
    \geq
    \frac{\lambda_i^q+\lambda_n^q\tau^2}{1+\tau^2}
    &=
    \lambda_i^q-(\lambda_i^q-\lambda_n^q)\frac{\tau^2}{1+\tau^2}\\
    &\geq
    (1-\epsilon^2)\lambda_i^q+\epsilon^2\lambda_n^q.
    \end{aligned}
\]
For RR error bound \eqref{eq:error-rr} with $q=1$, we obtain
\[
    \left|\lambda_i-\hat{\lambda}_i^{(\mathrm{RR})}\right|
    \leq
    (\lambda_i-\lambda_n)\epsilon^2
    \leq
    \lambda_i\epsilon^2,
\]
For \svdmethod\ with $q=2$, we have
\[
    \begin{aligned}
    \left|\lambda_i-\hat{\lambda}_i^{(\mathrm{SVD})}\right|
    \leq
    \lambda_i-\sqrt{(1-\epsilon^2)\lambda_i^2+\epsilon^2\lambda_n^2}
    &\leq
    \lambda_i(1-\sqrt{1-\epsilon^2})\\
    &=
    \frac{\lambda_i\epsilon^2}{2}+\frac{\lambda_i\epsilon^4}{2(1+\sqrt{1-\epsilon^2})^2},
    \end{aligned}
\]
which proves~\eqref{eq:error-svd}.

For \Nystrom, if $\lambda_{k+1}=0$, then $\image(A^{1/2}Q)=\image(U_1)=\image(A)$ and $A\langle Q\rangle=A$. Suppose that $\lambda_{k+1}>0$. Since $\Lambda_1^{1/2}$ is nonsingular, the image of $\mathcal E_i$ under $x\mapsto[\Lambda_1^{1/2};\Lambda_2^{1/2}Z]x$ has dimension $i$, i.e., $\dim([\Lambda_1^{1/2};\Lambda_2^{1/2}Z]x) = \dim(\mathcal{E}_i) = i$. Then, 
\begin{equation}
    \label{eq:abs-nys-maxmin}
    \begin{aligned}
    \hat{\lambda}_i^{(\mathrm{Nys})}
    =
    \rho_i(\Lambda;\mathcal{V})
    &\geq
    \min_{\substack{x\in\mathcal E_i\\\|x\|_2=1}}\frac{x^\top\Lambda_1^2x+(Zx)^\top\Lambda_2^2(Zx)}{x^\top\Lambda_1x+(Zx)^\top\Lambda_2(Zx)}\\
    &\geq
    \lambda_i-\max_{\substack{x\in\mathcal E_i\\\|x\|_2=1}}\frac{(Zx)^\top(\lambda_i\Lambda_2-\Lambda_2^2)(Zx)}{x^\top\Lambda_1x+(Zx)^\top\Lambda_2(Zx)},
    \end{aligned}
\end{equation}
where $x^\top\Lambda_1^2x\geq\lambda_i x^\top\Lambda_1x$. The assumption \eqref{eq:abs-local-condition} gives
\[
 \epsilon^2\leq\frac{\lambda_{k+1}}{\lambda_1}
 \leq\frac{\lambda_{k+1}}{\lambda_i}.
\]
Since $0\preceq\Lambda_2\preceq\lambda_{k+1}I\preceq\lambda_iI$, it follows that
\[
    \begin{aligned}
    &\lambda_{k+1}\lambda_i(1-\epsilon^2)I - \lambda_i\Lambda_2+\Lambda_2^2 + \lambda_{k+1}\epsilon^2\Lambda_2\\
    &=
    \lambda_i(\lambda_{k+1}I-\Lambda_2) + \Lambda_2^2 - \lambda_{k+1}\epsilon^2(\lambda_iI-\Lambda_2)\\
    &\succeq
    \lambda_i(\lambda_{k+1}I-\Lambda_2) + \Lambda_2^2 - \frac{\lambda_{k+1}^2}{\lambda_i} (\lambda_iI-\Lambda_2)\\
    &=
    \left(1-\frac{\lambda_{k+1}}{\lambda_i}\right) (\lambda_iI-\Lambda_2)(\lambda_{k+1}I-\Lambda_2) + \frac{\lambda_{k+1}}{\lambda_i}\Lambda_2^2 \succeq0,
    \end{aligned}
\]
Then, we have
\[
    \lambda_i\Lambda_2-\Lambda_2^2
    \preceq
    \lambda_{k+1}\epsilon^2\Lambda_2 + \lambda_{k+1}\lambda_i(1-\epsilon^2)I,
\]
where the first term is the trailing part of the denominator in \eqref{eq:abs-nys-maxmin}, and the second is controlled by~\eqref{eq:abs-angle-transfer}. Consequently,
\[
    \begin{aligned}
    (Zx)^\top(\lambda_i\Lambda_2-\Lambda_2^2)(Zx)
    &\leq
    \lambda_{k+1}\epsilon^2(Zx)^\top\Lambda_2(Zx) + \lambda_{k+1}\lambda_i(1-\epsilon^2)\|Zx\|_2^2\\
    &\leq
    \lambda_{k+1}\epsilon^2\left[(Zx)^\top\Lambda_2(Zx)+\lambda_i\right]\\
    &\leq
    \lambda_{k+1}\epsilon^2\left[(Zx)^\top\Lambda_2(Zx)+x^\top\Lambda_1x\right].
    \end{aligned}
\]
where the last two inequalities are due to $x^\top\Lambda_1x\geq\lambda_i$ and $(1-\epsilon^2)\|Zx\|_2^2\leq\epsilon^2$. Substituting this into~\eqref{eq:abs-nys-maxmin} gives $\hat{\lambda}_i^{(\mathrm{Nys})} \geq\lambda_i-\lambda_{k+1}\epsilon^2$ and the \Nystrom\ error bound \eqref{eq:error-nys-sharp}.

Without the condition \eqref{eq:abs-local-condition}, we instead drop the nonnegative trailing numerator in~\eqref{eq:abs-nys-maxmin}. For every unit vector $x\in\mathcal E_i$,
\[
    \frac{x^\top\Lambda_1^2x+(Zx)^\top\Lambda_2^2(Zx)}{x^\top\Lambda_1x+(Zx)^\top\Lambda_2(Zx)}
    \quad\geq
    \frac{\lambda_i x^\top\Lambda_1x}{x^\top\Lambda_1x+(Zx)^\top\Lambda_2(Zx)}
    \geq
    \frac{\lambda_i^2}{\lambda_i+\lambda_{k+1}\tau^2},
\]
since $x^\top\Lambda_1x\geq\lambda_i$ and $0\leq(Zx)^\top\Lambda_2(Zx)\leq\lambda_{k+1}\tau^2$. Then we obtain the "unconditional" \Nystrom\ bound \eqref{eq:error-nys} from \eqref{eq:abs-nys-maxmin}
\[
    \begin{aligned}
    \left|\lambda_i-\hat\lambda_i^{(\mathrm{Nys})}\right|
    \leq
    \frac{\lambda_i\lambda_{k+1}\tau^2}{\lambda_i+\lambda_{k+1}\tau^2}
    &\leq
    \frac{\lambda_{k+1}\epsilon^2}{1-(1-\lambda_{k+1}/\lambda_i)\epsilon^2}\\
    &=
    \lambda_{k+1}\epsilon^2+\frac{\lambda_{k+1}(1-\lambda_{k+1}/\lambda_i)\epsilon^4}{1-(1-\lambda_{k+1}/\lambda_i)\epsilon^2}.
    \end{aligned}
\]
The second inequality is due to $\tau^2\leq\epsilon^2/(1-\epsilon^2)$ and $f(t) =\lambda_i\lambda_{k+1}t/(\lambda_i+\lambda_{k+1}t)$ is nondecreasing for $t\geq0$.
\end{proof}

Theorem~\ref{thm:main-theorem} gives the error bounds for the leading $k$ eigenvalues in the case where the prescribed subspace $\image(Q)$ is exactly $k$-dimensional. In practice, the search subspace often has dimension $p>k$, although only the leading $k$ eigenvalues are sought. The following corollary uses the full $p$-dimensional search subspace and requires only that it captures the leading $k$-dimensional invariant subspace.

The structural representation of $Q$ in Theorem~\ref{thm:main-theorem} explicitly describes the behavior of $Q$ in $U$-coordinates. Since in the next corollary we allow $Q$ to have $p>k$ orthonormal columns, it is more convenient to express the subspace accuracy condition directly through
\begin{equation}
    \label{eq:oversample-subspace-condition}
    \sin\theta_{\max}(\leadingsubspace,\mathcal{S})
    =\|(I-\projector_{\mathcal{S}})U_1\|_2
    \leq\epsilon_{\mathcal{S}},\qquad \mathcal{S}=\image(Q), \dim(\mathcal{S}) = p.
\end{equation}
This condition measures how well the $p$-dimensional $\mathcal{S}$ approximates the $k$-dimensional target $\leadingsubspace$. The proof selects an accurate $k$-dimensional subspace inside $\mathcal{S}$, applies Theorem~\ref{thm:main-theorem} to it, and then transfers the error bounds to the full prescribed $p$-dimensional subspace.

The condition \eqref{eq:oversample-subspace-condition} is natural for oversampled subspace iteration, where oversampling directly leads to accelerated convergence~\cite[Thm.~5.2]{saad2011numerical}. In exact arithmetic, if $U_1^\top \Omega$ is full row rank and $\lambda_k>\lambda_{k+1}$, we can attain any prescribed $0<\epsilon_{\mathcal{S}}<1$ with sufficient iterations, where $\Omega$ is the initialization matrix. The full unrestarted block Krylov subspace containing the subspace generated by the corresponding subspace iteration inherits this guarantee; see the detailed discussion after Corollary~\ref{corollary:oversampling-corollary} for more details.

\begin{corollary}
\label{corollary:oversampling-corollary}
Let $A\succeq0$, $1\leq k<n$, and $\lambda_k>0$. Suppose $Q\in\realR^{n\times p}$ is orthonormal, $k\leq p\leq n$, and $\mathcal{S} := \image(Q)$ satisfies
\begin{equation}
    \label{eq:subspace-accuracy-assumption-oversampling}
    \|(I-\projector_{\mathcal{S}})U_1\|_2 \leq \epsilon_{\mathcal{S}}<1.
\end{equation}
Then, for $i=1,\ldots,k$,
\begin{equation}
    \label{eq:error-bound-RR-oversampling}
    \left|\lambda_i-\lambda_i(Q^\top A Q)\right|
    \leq
    (\lambda_i-\lambda_n)\epsilon_{\mathcal{S}}^2
    \leq
    \lambda_i\epsilon_{\mathcal{S}}^2,
\end{equation}
\begin{equation}
    \label{eq:error-bound-svd-oversampling}
    \left|\lambda_i-\sigma_i(A Q)\right|
    \leq
    \lambda_i-\sqrt{(1-\epsilon_{\mathcal{S}}^2)\lambda_i^2+\epsilon_{\mathcal{S}}^2\lambda_n^2}
    \leq
    \frac{\lambda_i\epsilon_{\mathcal{S}}^2}{2}+\frac{\lambda_i\epsilon_{\mathcal{S}}^4} {2(1+\sqrt{1-\epsilon_{\mathcal{S}}^2})^2},
\end{equation}
and
\begin{equation}
    \label{eq:error-bound-nys-sharp-oversampling}
    \left|\lambda_i-\lambda_i(A \langle Q \rangle)\right|
    \leq
    \frac{\lambda_{k+1}\epsilon_{\mathcal{S}}^2}{1-(1-\lambda_{k+1}/\lambda_i)\epsilon_{\mathcal{S}}^2}
    =
    \lambda_{k+1}\epsilon_{\mathcal{S}}^2+\frac{\lambda_{k+1}(1-\lambda_{k+1}/\lambda_i)\epsilon_{\mathcal{S}}^4}{1-(1-\lambda_{k+1}/\lambda_i)\epsilon_{\mathcal{S}}^2}.
\end{equation}
If $\lambda_{k+1}>0$ and $\epsilon_{\mathcal{S}}^2\leq\lambda_{k+1}/\lambda_1$, then
\begin{equation}
    \label{eq:error-bound-nys-general-oversampling}
    \left|\lambda_i-\lambda_i(A \langle Q \rangle)\right|\leq\lambda_{k+1}\epsilon_{\mathcal{S}}^2,
    \qquad i=1,\ldots,k.
\end{equation}
\end{corollary}

\begin{proof}
We will construct an orthonormal basis $X$ of a $k$-dimensional subspace contained in $\mathcal{S}$ that satisfies the structural assumption of Theorem~\ref{thm:main-theorem} on the orthonormal basis matrix with $\epsilon=\epsilon_{\mathcal{S}}$. The resulting bounds can be directly applied to $Q$ since $\image(X) \subseteq \mathcal{S}$ and $\image(A^{1/2}X) \subseteq \image(A^{1/2}Q)$, and the Ritz values on a subspace cannot be larger than the Ritz values on a larger-dimensional subspace that contains it.

Project $U_1$ onto $\mathcal{S}$ and let $G:=(\projector_{\mathcal{S}}U_1)^\top (\projector_{\mathcal{S}}U_1) =U_1^\top\projector_{\mathcal{S}}U_1$. The subspace accuracy assumption \eqref{eq:subspace-accuracy-assumption-oversampling} implies
\[
    \epsilon_{\mathcal{S}}^2I_k
    \succeq
    I_k-G
    =
    \left(((I-\projector_{\mathcal{S}})U_1\right))^\top \left(((I-\projector_{\mathcal{S}})U_1\right))
    \succeq
    0.
\]
Then, $G\succeq(1-\epsilon_{\mathcal{S}}^2)I_k\succ0$. We construct the orthonormal basis $X:=\projector_{\mathcal{S}}U_1G^{-1/2}$ that satisfies $X^\top X=G^{-1/2}GG^{-1/2}=I_k$, $\image(X)\subseteq\mathcal{S}$,
\[
    U_1^\top X=G^{1/2},\qquad
    (U_2^\top X)^\top(U_2^\top X)=I_k-G.
\]
To show that $X$ satisfies the structural assumption in Theorem~\ref{thm:main-theorem}, for $\epsilon_{\mathcal{S}}>0$, let
\[
    M:=\frac{G^{1/2}-I_k}{\epsilon_{\mathcal{S}}^2},
    \qquad
    N:=\frac{U_2^\top X}{\epsilon_{\mathcal{S}}}.
\]
Then, 
\[
    \|N\|_2 \leq 1,
    \qquad
    \|M\|_2 \leq \frac{1-\sqrt{1-\epsilon_{\mathcal{S}}^2}}{\epsilon_{\mathcal{S}}^2}=\frac{1}{1+\sqrt{1-\epsilon_{\mathcal{S}}^2}}\leq1,
\]
and therefore $X$ satisfies \eqref{eq:Q-in-U-coordinate-thm} such that $X=U_1+\epsilon_{\mathcal{S}}^2U_1M+\epsilon_{\mathcal{S}}U_2N$.

We then bound the Ritz values on $\image(X)$ and $\image(A^{1/2}X)$ above by the corresponding Ritz values on $\image(Q)$ and $\image(A^{1/2}Q)$, respectively. Since $\lambda_k > 0$, $U_1^\top A^{1/2}X=\Lambda_1^{1/2}G^{1/2}$ is nonsingular, we have $\dim(\image(A^{1/2}X)) = k$. By $\image(X)\subseteq\image(Q)$, $\image(A^{1/2}X)\subseteq\image(A^{1/2}Q)$ and Theorem~\ref{thm:main-theorem}, for $i=1,\ldots,k$, we obtain
\[
    \lambda_i(X^\top AX)
    \leq\lambda_i(Q^\top AQ),
    \qquad
    \sigma_i(AX)
    \leq\sigma_i(AQ),
    \qquad
    \lambda_i(A\langle X\rangle)
    \leq\lambda_i(A\langle Q\rangle),
\]
\[
    \left|\lambda_i-\lambda_i(Q^\top AQ)\right|
    \leq
    \left|\lambda_i-\lambda_i(X^\top AX)\right|,
    \qquad
    \left|\lambda_i-\sigma_i(AQ)\right|
    \leq
    \left|\lambda_i-\sigma_i(AX)\right|,
\]
\[
    \left|\lambda_i-\lambda_i(A\langle Q \rangle)\right|
    \leq
    \left|\lambda_i-\lambda_i(A\langle X \rangle)\right|,
\]
which completes the proof.
\end{proof}

Since the RR, \svdmethod\ and \Nystrom\ error bounds in Theorem~\ref{thm:main-theorem} and Corollary~\ref{corollary:oversampling-corollary} are sharp, we now can quantify the error improvement of one method over one another. Throughout the following discussion, let $\epsilon$ denote an upper bound for $\sin\theta_{\max}(\leadingsubspace,\image(Q))$, no matter whether $\image(Q)$ has dimension $k$ or a larger dimension $p\geq k$. This covers the settings of both Theorem~\ref{thm:main-theorem} and Corollary~\ref{corollary:oversampling-corollary}, with $\epsilon=\epsilon_{\mathcal S}$ in the latter. For $0<\epsilon<1$, $\lambda_i>\lambda_n$ and $\epsilon^2\leq\lambda_{k+1}/\lambda_1$, write
\begin{equation}
    \label{eq:bound-ratio-definitions}
    B_i^{(\mathrm{RR})}:=(\lambda_i-\lambda_n)\epsilon^2,
    \quad
    B_i^{(\mathrm{SVD})}:=\lambda_i-\sqrt{(1-\epsilon^2)\lambda_i^2+\epsilon^2\lambda_n^2},
    \quad
    B_i^{(\mathrm{Nys})}:=\lambda_{k+1}\epsilon^2.
\end{equation}
The corresponding ratios are
\begin{equation}
    \label{eq:improvement-ratios}
    \begin{aligned}
    C_{1,i}
    :=
    \frac{B_i^{(\mathrm{Nys})}}{B_i^{(\mathrm{RR})}} &= \frac{\lambda_{k+1}}{\lambda_i-\lambda_n},\\
    C_{2,i}
    :=
    \frac{B_i^{(\mathrm{Nys})}}{B_i^{(\mathrm{SVD})}} &= \frac{\lambda_{k+1} \left(\lambda_i+\sqrt{(1-\epsilon^2)\lambda_i^2+\epsilon^2\lambda_n^2}\right)}{\lambda_i^2-\lambda_n^2},\\
    C_{3,i}
    :=
    \frac{B_i^{(\mathrm{SVD})}}{B_i^{(\mathrm{RR})}} &= \frac{\lambda_i+\lambda_n}{\lambda_i+\sqrt{(1-\epsilon^2)\lambda_i^2+\epsilon^2\lambda_n^2}}.
    \end{aligned}
\end{equation}
For a fixed spectrum, their limits as $\epsilon\to0$ are
\begin{equation}
    \label{eq:improvement-ratios-asymptotic}
    C_{1,i} \to \frac{\lambda_{k+1}}{\lambda_i-\lambda_n},\qquad
    C_{2,i} \to \frac{2\lambda_{k+1}\lambda_i}{\lambda_i^2-\lambda_n^2},\qquad
    C_{3,i} \to \frac{\lambda_i+\lambda_n}{2\lambda_i}.
\end{equation}
When $\lambda_n/\lambda_i$ is negligible and $\epsilon$ is small, these are approximately
\begin{equation}
    \label{eq:small-tail-comparison-scales}
    \frac{\lambda_{k+1}}{\lambda_i},
    \qquad
    \frac{2\lambda_{k+1}}{\lambda_i},
    \qquad
    \frac{1}{2},
\end{equation}
which quantify the error improvement of \Nystrom\ over RR, \Nystrom\ over \svdmethod\ and \svdmethod\ over RR, respectively. Thus, a small $\lambda_{k+1}/\lambda_i$ favors the \Nystrom\ bound for target $i$ and a small $\lambda_{k+1}/\lambda_k$ gives this advantage uniformly for $i\leq k$.

If the spectral condition $\epsilon^2\leq\lambda_{k+1}/\lambda_1$ does not hold, we can simply replace $B_i^{(\mathrm{Nys})}$ in the ratios $C_{1,i}$ and $C_{2,i}$ by $B_i^{(\mathrm{Nys})}/\left(1-(1-\lambda_{k+1}/\lambda_i)\epsilon^2\right)$. The asymptotic improvement ratios in \eqref{eq:improvement-ratios-asymptotic} and \eqref{eq:small-tail-comparison-scales} still hold. If $\lambda_i=\lambda_n$, the corresponding eigenvalue approximation errors vanish and ratios with a zero denominator are not defined. If $\lambda_{k+1}=0$, \Nystrom\ computes exact eigenvalue of $A$.

Note that the above improvement ratios $C_{l,i}$ with $l=1,2,3$ are obtained by comparing the sharp upper error bounds from Theorem~\ref{thm:main-theorem} and Corollary~\ref{corollary:oversampling-corollary}, rather than directly comparing the exact eigenvalue approximation errors. Nevertheless, the spectral improvement ratios $C_{l,i}$ can also be observed in the exact errors. For example, consider a flat trailing spectrum and an $n\times k$ orthonormal basis whose principal angles with $\leadingsubspace$ are all equal such that
\[
    \Lambda_2=\lambda_{k+1}I_{n-k},
    \qquad
    Q_{\epsilon}=\sqrt{1-\epsilon^2}\,U_1+\epsilon U_2W,
    \qquad
    W^\top W=I_k.
\]
This basis satisfies structural condition in Theorem~\ref{thm:main-theorem} with $M=-I_k/(1+\sqrt{1-\epsilon^2})$ and $N=W$. Direct calculation gives
\[
    \hat{\lambda}_i^{(\mathrm{RR})}
    =
    (1-\epsilon^2)\lambda_i+\epsilon^2\lambda_{k+1},
    \qquad
    \hat{\lambda}_i^{(\mathrm{SVD})}
    =
    \sqrt{(1-\epsilon^2)\lambda_i^2+\epsilon^2\lambda_{k+1}^2},
\]
\[
    \hat{\lambda}_i^{(\mathrm{Nys})}
    =
    \frac{(1-\epsilon^2)\lambda_i^2+\epsilon^2\lambda_{k+1}^2}{(1-\epsilon^2)\lambda_i+\epsilon^2\lambda_{k+1}}
    =
    \frac{(\hat{\lambda}_i^{(\mathrm{SVD})})^2}{\hat{\lambda}_i^{(\mathrm{RR})}}.
\]
If $\lambda_{k+1}\ll\lambda_i$, as $\epsilon \to 0$, the exact error ratios converge to
\[
    \frac{\lambda_{k+1}}{\lambda_i},
    \qquad
    \frac{2\lambda_{k+1}}{\lambda_i+\lambda_{k+1}},
    \qquad
    \frac{\lambda_i+\lambda_{k+1}}{2\lambda_i}.
\]

For a particular target index $i$, the ideal spectral decaying has $\lambda_{k+1}/\lambda_i\ll1$. For the improvement ratios to be uniformly small over $ i\leq k$, we need a stronger decaying condition $\lambda_{k+1}/\lambda_k\ll1$. We give some examples to describe the ideal spectral decays.
\paragraph{Geometric decay}
Let $\lambda_j=\lambda_1q^{j-1}$, where $0<q<1$. Then
\[
    \frac{\lambda_{k+1}}{\lambda_i}=q^{k+1-i},
    \qquad
    \frac{\lambda_{k+1}}{\lambda_k}=q,
    \qquad
    \frac{\lambda_{k+1}}{\lambda_1}=q^k.
\]
The first ratio is small for target indices $i$ which are sufficiently far from $k$, but a uniformly small ratio across all $k$ targets requires $q$ itself to be small. Note that increasing $k$ does not help to improve the ratio at $i=k$. The sharper \Nystrom\ error bounds \eqref{eq:error-nys-sharp} and \eqref{eq:error-bound-nys-general-oversampling} can be obtained when $\epsilon^2\leq q^k$, which requires a much more accurate subspace than the condition $\epsilon<1$ for the general bound.

\paragraph{Algebraic decay}
Let $\lambda_j=\lambda_1j^{-s}$, where $s>0$. Then
\[
    \frac{\lambda_{k+1}}{\lambda_i}
    =\left(\frac{i}{k+1}\right)^s,
    \qquad
    \frac{\lambda_{k+1}}{\lambda_k}
    =\left(\frac{k}{k+1}\right)^s,
    \qquad
    \frac{\lambda_{k+1}}{\lambda_1}=(k+1)^{-s}.
\]
The first few target indices $i$ can have ideal \Nystrom\ improvement over both RR and \svdmethod. However, for fixed $s$, the ratio $\lambda_{k+1}/\lambda_k$ at $i=k$ tends to one as $k$ increases. Thus, it is hard for algebraic decay to obtain uniformly significant error improvement for all $i=1,. .., k$. The sharper \Nystrom\ error bounds can be obtained when $\epsilon^2\leq(k+1)^{-s}$.

\paragraph{A separate leading block and a small tail}
A uniformly significant error improvement can occur when there is a drop from $\lambda_k$ to $\lambda_{k+1}$. For example, let $k\geq2$ and
\[
    1=\lambda_1\geq\cdots\geq\lambda_k=\tfrac12,
    \qquad
    \lambda_{k+1}=10^{-3}\geq\cdots\geq\lambda_n\geq0.
\]
In this example we have $\lambda_{k+1}/\lambda_i\leq0.002$ for every $i\leq k$ and the condition for obtaining the sharper \Nystrom\ error bounds is $\epsilon^2\leq10^{-3}$. The spectrum decaying from $\lambda_{k+1}$ to $\lambda_n$ can be flat, geometric, or algebraic.

The examples above separate good spectral decays from the applicability of Theorem~\ref{thm:main-theorem} and Corollary~\ref{corollary:oversampling-corollary}. The eigenvalue error bounds only require the prescribed subspace to satisfy the accuracy condition but neither spectral decay, a significant drop from $\lambda_k$ to $\lambda_{k+1}$, nor the simplicity requirement of leading eigenvalues. Spectral decay only influences the size of the bounds but does not itself guarantee we have an accurate search subspace. The sharper \Nystrom\ bound additionally requires the given subspace to satisfy $\epsilon^2\leq\lambda_{k+1}/\lambda_1$, which might be more restrictive if $\lambda_{k+1}$ is small. However, without this condition, the "unconditional" \Nystrom\ error bounds \eqref{eq:error-nys} and \eqref{eq:error-bound-nys-sharp-oversampling} are still $O(\lambda_{k+1} \epsilon^2)$ such that
\[
    \left|\lambda_i-\lambda_i(A\langle Q\rangle)\right|
    \leq
    \frac{\lambda_{k+1}\epsilon^2}
    {1-(1-\lambda_{k+1}/\lambda_i)\epsilon^2}
    \leq
    \frac{\lambda_{k+1}\epsilon^2}{1-\epsilon^2}.
\]

Here we discuss when $Q$ can be obtained from the subspace approximation algorithms since the eigenvalue error bounds are conditional on the accuracy of the given subspace. The following discussion holds in exact arithmetic. For subspace iteration algorithms, let $\Omega\in\realR^{n\times b}$ be the initialization matrix with $k\leq b\leq n$ and
\[
    \Omega_1:=U_1^\top\Omega,
    \qquad
    \Omega_2:=U_2^\top\Omega,
\]
where $\Omega_1$ is rank-$k$ with probability one for standard Gaussian $\Omega$ \cite[Section~10.2]{halko2011finding}. Let $\mathcal{S}_m:=\image(A^m\Omega)$ with $k \leq p:=\dim(\mathcal{S}_m) \leq b$, where $m\geq0$ and $A^0:=I_n$. Since $\lambda_k>0$ and $\Omega_1\Omega_1^\dagger=I_k$,
\[
    \image(A^m\Omega\Omega_1^\dagger\Lambda_1^{-m})
    =
    \image
    \left(
    U
    \begin{bmatrix}
    I_k\\
    Z_m
    \end{bmatrix}
    \right)
    \subseteq
    \mathcal{S}_m,
    \qquad
    Z_m:=\Lambda_2^m\Omega_2\Omega_1^\dagger\Lambda_1^{-m}.
\]
Then, the tangent of the largest principal angle satisfies \cite[Section~5.1]{saad2011numerical}
\begin{equation}
    \label{eq:tangent-bound-subspace-iteration}
    \tan\theta_{\max}(\leadingsubspace,\mathcal{S}_m)
    \leq\|Z_m\|_2
    \leq
    \left(\frac{\lambda_{k+1}}{\lambda_k}\right)^m\eta_0,
    \qquad
    \eta_0:=\|\Omega_2\Omega_1^\dagger\|_2.
\end{equation}
The first inequality is an equality when $b=k$. For $\lambda_{k+1}>0$, let $\tau_m:=(\lambda_{k+1}/\lambda_k)^m\eta_0$. It follows that
\begin{equation}
    \label{eq:sine-bound-subspace-iteration}
    \sin\theta_{\max}(\leadingsubspace,\mathcal{S}_m)
    =
    \|(I-\projector_{\mathcal{S}_m})U_1\|_2
    \leq
    \frac{\tau_m}{\sqrt{1+\tau_m^2}}.
\end{equation}
If $0<\lambda_{k+1}<\lambda_k$ and $\eta_0>0$, then for any prescribed $0<\epsilon<1$ it suffices to have
\begin{equation}
    \label{eq:sufficient-iteration-count-subspace-iteration}
    m
    \geq
    \max\left\{0, \left\lceil\frac{\log\left(\eta_0\sqrt{1-\epsilon^2}/\epsilon\right)}{\log(\lambda_k/\lambda_{k+1})}\right\rceil\right\}
\end{equation}
to guarantee
\[
    \sin\theta_{\max}(\leadingsubspace,\mathcal{S}_m) \leq \epsilon.
\]
In particular, when $b=k$, $\mathcal{S}_m$ has dimension $k$. In this case, suppose $Q_m$ is an orthonormal basis of $\mathcal{S}_m$ generated by subspace iteration. After an orthogonal change of basis by $V_m \in \realR^{k \times k}$, the structural assumption in Theorem~\ref{thm:main-theorem} is satisfied by
\[
    Q_m V_m = U_1 + \epsilon^2U_1M_m + \epsilon U_2N_m,
    \qquad
    \|M_m\|_2,\|N_m\|_2\leq1.
\]
Here $p$ may be smaller than $b$ if $A$ is singular. Note that we do not require the subspace accuracy for the leading $p$-dimensional invariant subspace. For the same initialization matrix, consider the full unrestarted block Krylov subspace
\[
    \mathcal{K}_{m+1}(A,\Omega):=\image([\Omega,A\Omega,\ldots,A^m\Omega]).
\]
Since $\mathcal{S}_m\subseteq\mathcal{K}_{m+1}(A,\Omega)$,
\begin{equation}
    \label{eq:sine-bound-Krylov-subspace}
    \|(I-\projector_{\mathcal{K}_{m+1}(A,\Omega)})U_1\|_2
    \leq
    \|(I-\projector_{\mathcal{S}_m})U_1\|_2
    \leq
    \frac{\tau_m}{\sqrt{1+\tau_m^2}}.
\end{equation}
Thus, \eqref{eq:sufficient-iteration-count-subspace-iteration} is also a sufficient iteration count for the Krylov subspace methods to satisfy the subspace accuracy assumption in Corollary~\ref{corollary:oversampling-corollary}. For the sharper \Nystrom\ bounds \eqref{eq:error-nys-sharp} and \eqref{eq:error-bound-nys-sharp-oversampling}, one may choose $\epsilon_{\mathrm{prior}}$ with $\epsilon_{\mathrm{priori}}^2\leq\lambda_{k+1}/\lambda_1$ to obtain a prior sufficient iteration counts.

Rather than showing error bounds to imply the improvement ratios as in Theorem~\ref{thm:main-theorem}, an alternative way exists for directly comparing the accuracy improvement for three methods, that is, to expand the Ritz value representations in \eqref{eq:ritz-representation-in-U-coordinate} at $Z=0$ for RR and \svdmethod, with $Z_{\mathrm{Nys}}=\Lambda_2^{1/2}Z\Lambda_1^{-1/2}$ for \Nystrom. For perturbations of the search subspace toward $u_{k+1}$, the leading quadratic (i.e., $\tau^2$ or $\epsilon^2$) error coefficients for RR, \svdmethod, and \Nystrom\ are of order $\lambda_i$, $\lambda_i/2$, and $\lambda_{k+1}$ when $\lambda_{k+1}/\lambda_i\ll1$, respectively.

\begin{remark}[Relation to existing results]
The advantage of \Nystrom\ over \svdmethod\ is known for matrix approximation~\cite[Lemma~5.2]{tropp2023randomized} and eigenvalue ordering~\cite[Theorem~3.7]{persson2025algorithm}. Our results complement these comparisons with eigenvalue error bounds given a search subspace for RR, \svdmethod\ and \Nystrom. Moreover, we quantify the accuracy improvement of one of the three methods over the others. In particular, when $\lambda_{k+1} \ll \lambda_{i}$, the accuracy improvement ratios of \Nystrom\ over RR and \svdmethod\ tend to be significant, and that of \svdmethod\ over RR tends to be $1/2$. The RR bound in Theorem~\ref{thm:main-theorem} is classical \cite{knyazev1985sharp}. Our work further provides the \svdmethod\ and \Nystrom\ bounds and compares all three to quantify the potential accuracy gains. For small $\epsilon$ and $\lambda_n\ll\lambda_i$, the \Nystrom\ error bound is approximately $\lambda_i/\lambda_{k+1}$ and $\lambda_i/(2\lambda_{k+1})$ times smaller than those of RR and \svdmethod, respectively, whereas the \svdmethod\ error bound is approximately twice as small as that of RR. Musco and Musco~\cite[Theorem~1]{musco2015randomized} analyze randomized subspace iteration and block Krylov methods for approximating singular values and singular vectors of a general matrix $A \in \realR^{n \times d}$. In the PSD case, they provide a high-probability, per-vector guarantee $\lambda_i^2-\sigma_i(AQ)^2\leq\delta\lambda_{k+1}^2$, where $\delta$ is an algorithmic tolerance rather than a principal angle sine tolerance. In contrast, Theorem~\ref{thm:main-theorem} and Corollary~\ref{corollary:oversampling-corollary} provide deterministic extraction error bounds that depend on the principal angle to the target invariant subspace. Our bounds do not supersede the algorithmic guarantees.
\end{remark}

\section{Numerical experiments}
\label{sec:Numerical-illustrations-section}
Here we conduct experiments to illustrate the theoretical findings. All experiments use $n=1000$ and diagonal PSD matrices $A=\diag(\lambda_1,\ldots,\lambda_n)$ with $\lambda_1=1$.

\subsection{Leading and trailing error ordering}
We give numerical illustrations for the eigenvalue error ordering discussion in Section~\ref{subsec:eigenvalue-error-ordering-subsection} before showing the leading eigenvalue error bounds. Let $k=200$ and consider the following geometric, algebraic and linear spectra for $i=1,\ldots,n$
\[  
    \begin{aligned}
        \tilde{\lambda}_i^{\mathrm{geo}}
        &=10^{-20(i-1)/(n-1)},\\
        \tilde{\lambda}_i^{\mathrm{pow}}
        &=10^{-20}+(1-10^{-20})
        \left(\frac{n-i}{n-1}\right)^4,\\
        \tilde{\lambda}_i^{\mathrm{lin}}
        &=1-(1-10^{-20})\frac{i-1}{n-1}.
    \end{aligned}
\]
For these spectra, we set $\lambda_i=0$ whenever the corresponding eigenvalue $\widetilde\lambda_i$ is below $10^{-15}$. For leading eigenvalues, we construct the orthonormal basis $Q=\operatorname{orth}(A\Omega)$ with Gaussian $\Omega\in\realR^{n\times k}$. The upper plots of Figure~\ref{fig:eigenvalue-error-ordering} illustrate the leading error ordering. For trailing targets, we consider the algebraic and linear spectra and construct $Q=\operatorname{orth}([u_{n-k+1},\ldots,u_n]+0.01G)$ with $G_{ij}\sim\mathcal N(0,1/n)$. The lower plots compare direct extraction with the shifted remedy discussed in Section~\ref{subsec:eigenvalue-error-ordering-subsection}, using $B=I-A$ and the same orthonormal basis $Q$. We assume the valid shift $\gamma=1$ is known from these spectra. The direct trailing eigenvalue error ordering reverses. The shift remedy improves the \Nystrom\ eigenvalue approximations. The RR trailing eigenvalue approximations are unchanged in exact arithmetic.
\begin{figure}[htbp]
\centering
\includegraphics[scale=0.28]{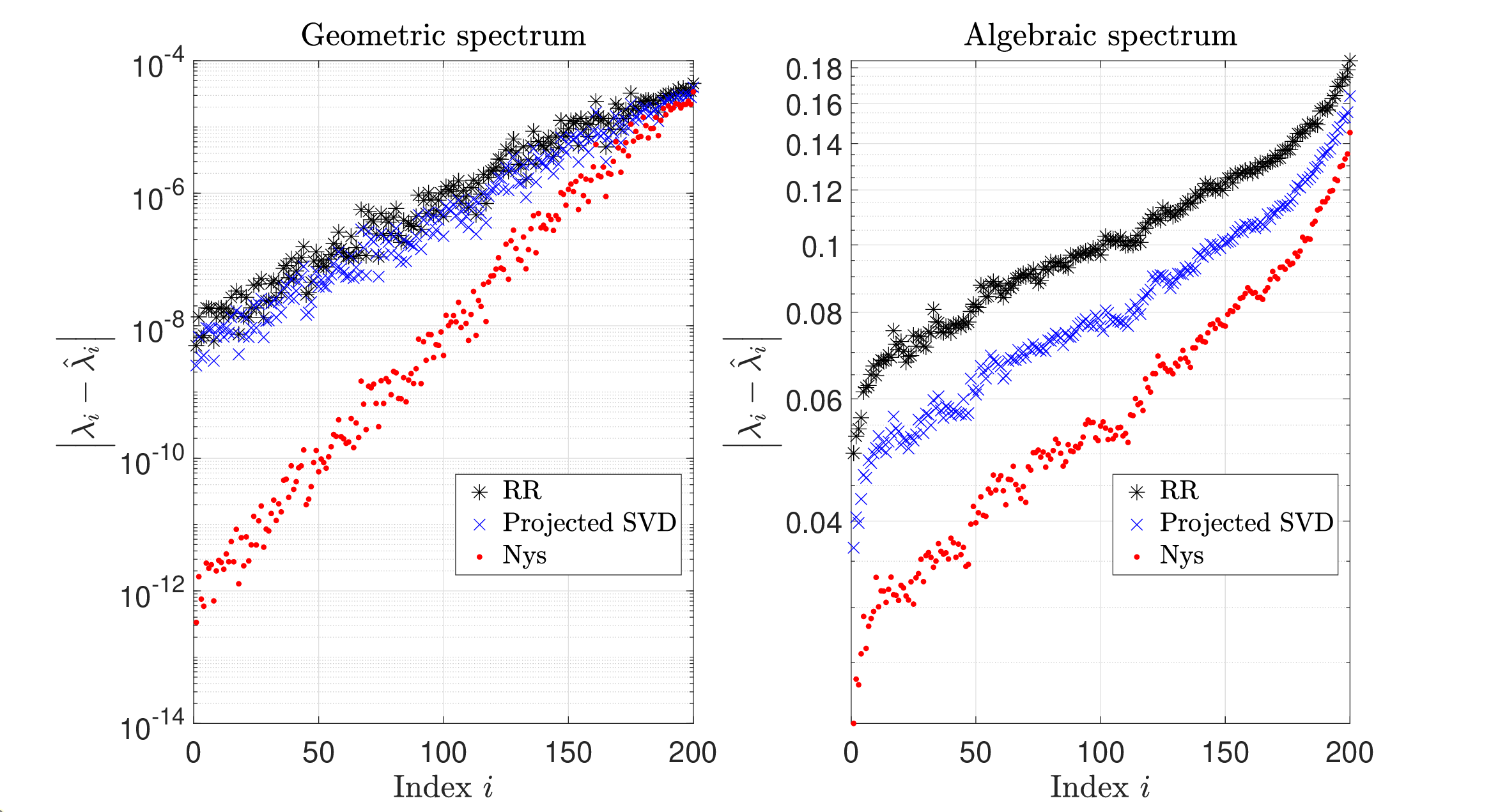}\\[2pt]
\includegraphics[scale=0.28]{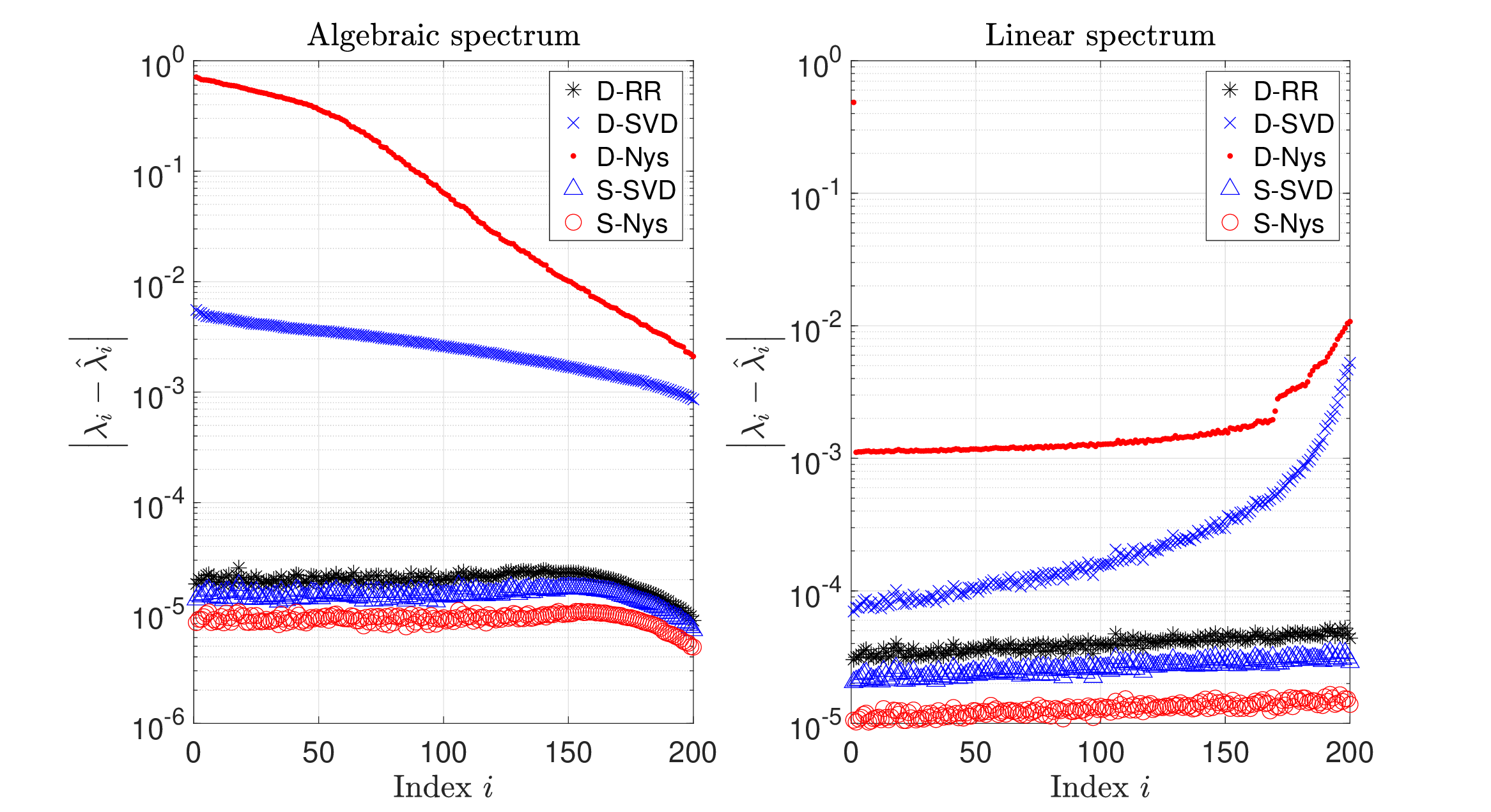}
\caption{Eigenvalue error ordering, $n=1000$, $k=200$. Top: leading targets for geometric (left) and algebraic (right) spectra. Bottom: trailing targets for linear (left) and algebraic (right) spectra, before and after shifting by $\gamma=1$. "D-" and "S-" denote "Direct" and "Shifted", respectively. Shifted RR coincides with direct RR and so is omitted.}
\label{fig:eigenvalue-error-ordering}
\end{figure}

\subsection{Leading eigenvalue error bounds}
To illustrate Theorem~\ref{thm:main-theorem} and Corollary~\ref{corollary:oversampling-corollary}, we let $k=20$ instead and consider the following spectrum with well-separated leading eigenvalues
\[
    \lambda_i=\begin{cases}
    10^{-(i-1)/19},&1\leq i\leq20,\\
    10^{-3-3(i-21)/979},&21\leq i\leq1000.
    \end{cases}
\]
Let $W \in \realR^{n \times k}$ be orthonormal such that $[W, W_{\perp}]$ has all zeros in the first row, which is generated by orthonormalizing an $(n-1) \times n$ random Gaussian matrix and adding a zero row. Let the $U$-coordinate representation of the orthonormal basis matrices be
\[
    Q_{20}=\begin{bmatrix}\sqrt{1-\epsilon^2}I_k\\\epsilon W\end{bmatrix},
    \quad
    Q_{40}=\begin{bmatrix}
    \sqrt{1-\epsilon^2}I_k&0\\\epsilon W&W_{\perp}
    \end{bmatrix},
\]
which satisfy $\sin \theta_{\max} (\leadingsubspace, \image(Q_{20})) = \sin \theta_{\max} (\leadingsubspace, \image(Q_{40}))=\epsilon$, where $\leadingsubspace:=\image([u_1, \ldots, u_{20}])$. Figure~\ref{fig:eigenvalue-error-and-bound} illustrates the computed eigenvalue errors and the corresponding RR, \svdmethod\ and \Nystrom\ bounds, with $\epsilon=0.01$ and $\epsilon=0.3$. Since $\epsilon=0.01$ satisfies $\epsilon^2 \leq \lambda_{21}/\lambda_1 =10^{-3}$, we plot the \Nystrom\ bounds by $\lambda_{21}\epsilon^2$, i.e., \eqref{eq:error-nys-sharp} and \eqref{eq:error-bound-nys-general-oversampling}. For $\epsilon=0.3$, the \Nystrom\ bounds is \eqref{eq:error-nys} and \eqref{eq:error-bound-nys-sharp-oversampling}.
\begin{figure}[htbp]
\centering
\includegraphics[width=.5\textwidth]{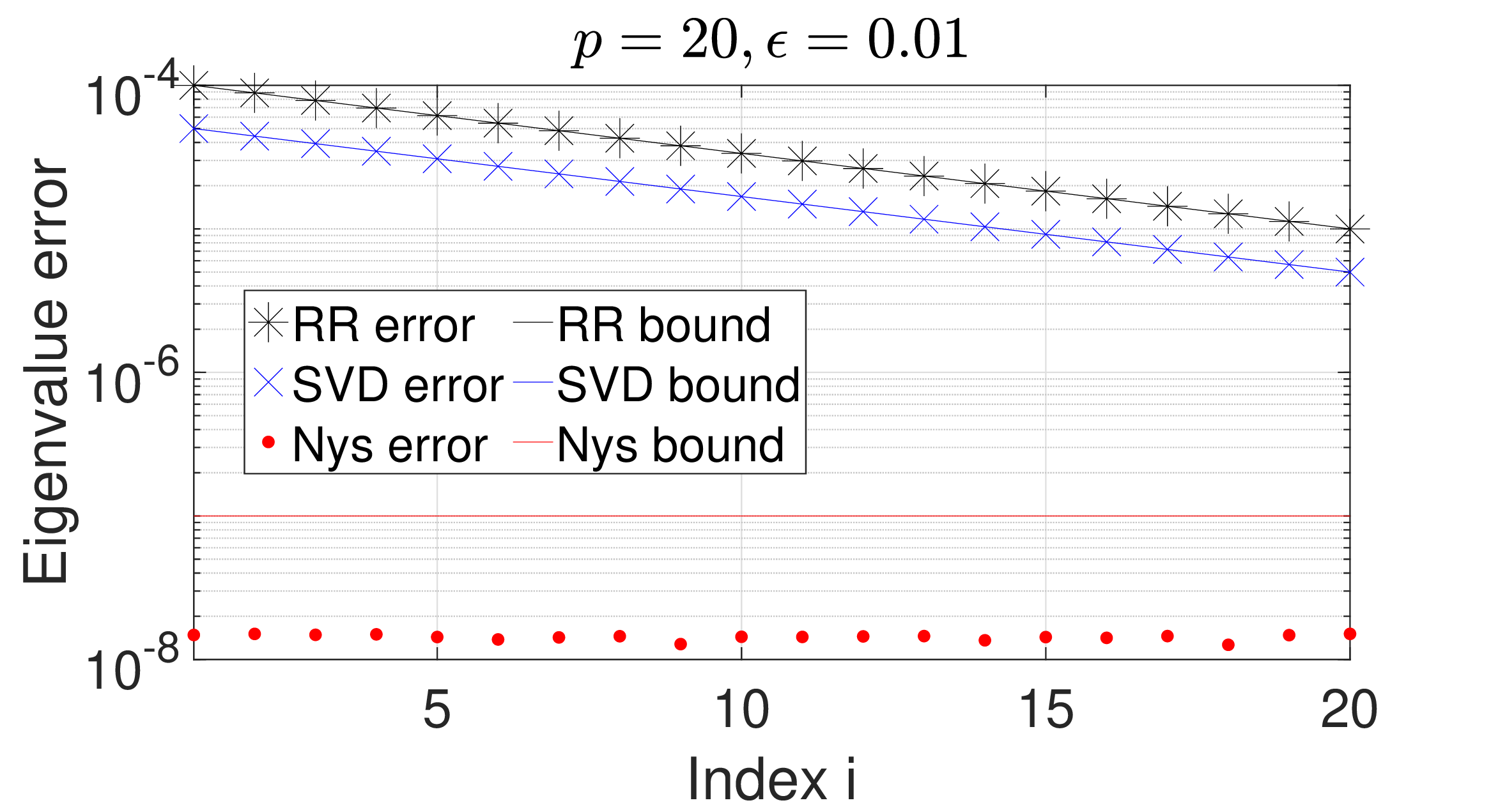}\hfill
\includegraphics[width=.5\textwidth]{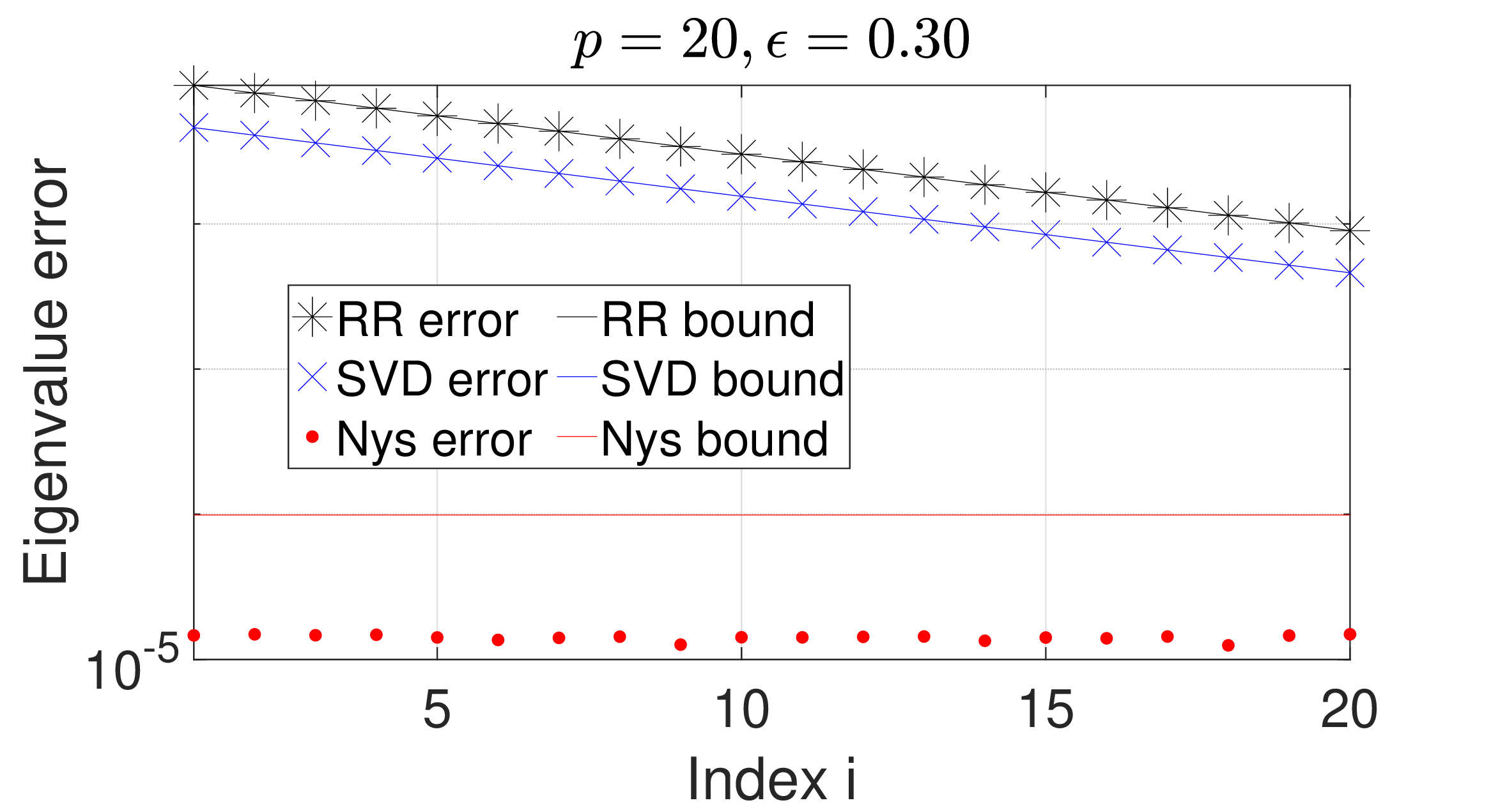}\\[2pt]
\includegraphics[width=.5\textwidth]{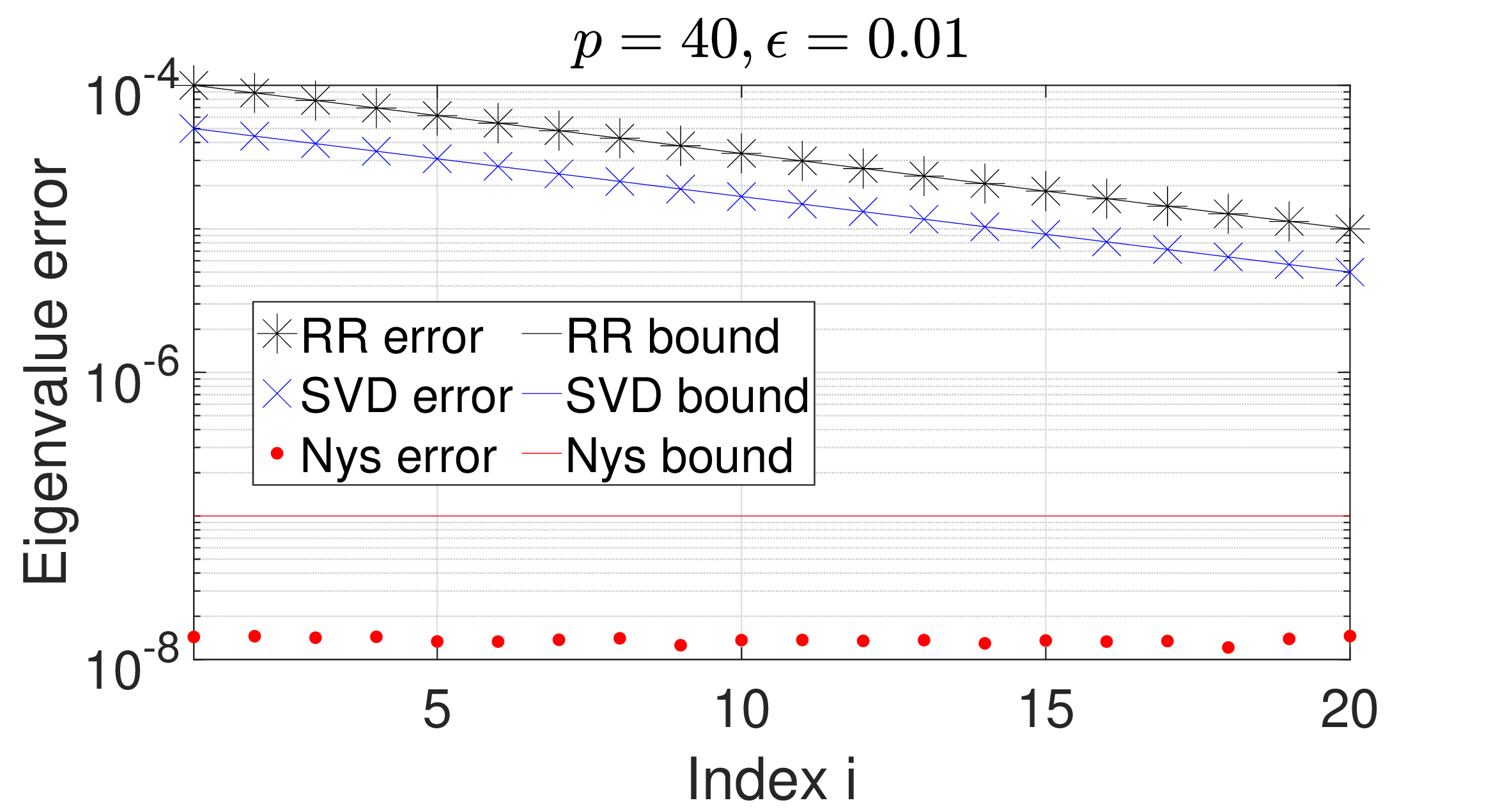}\hfill
\includegraphics[width=.5\textwidth]{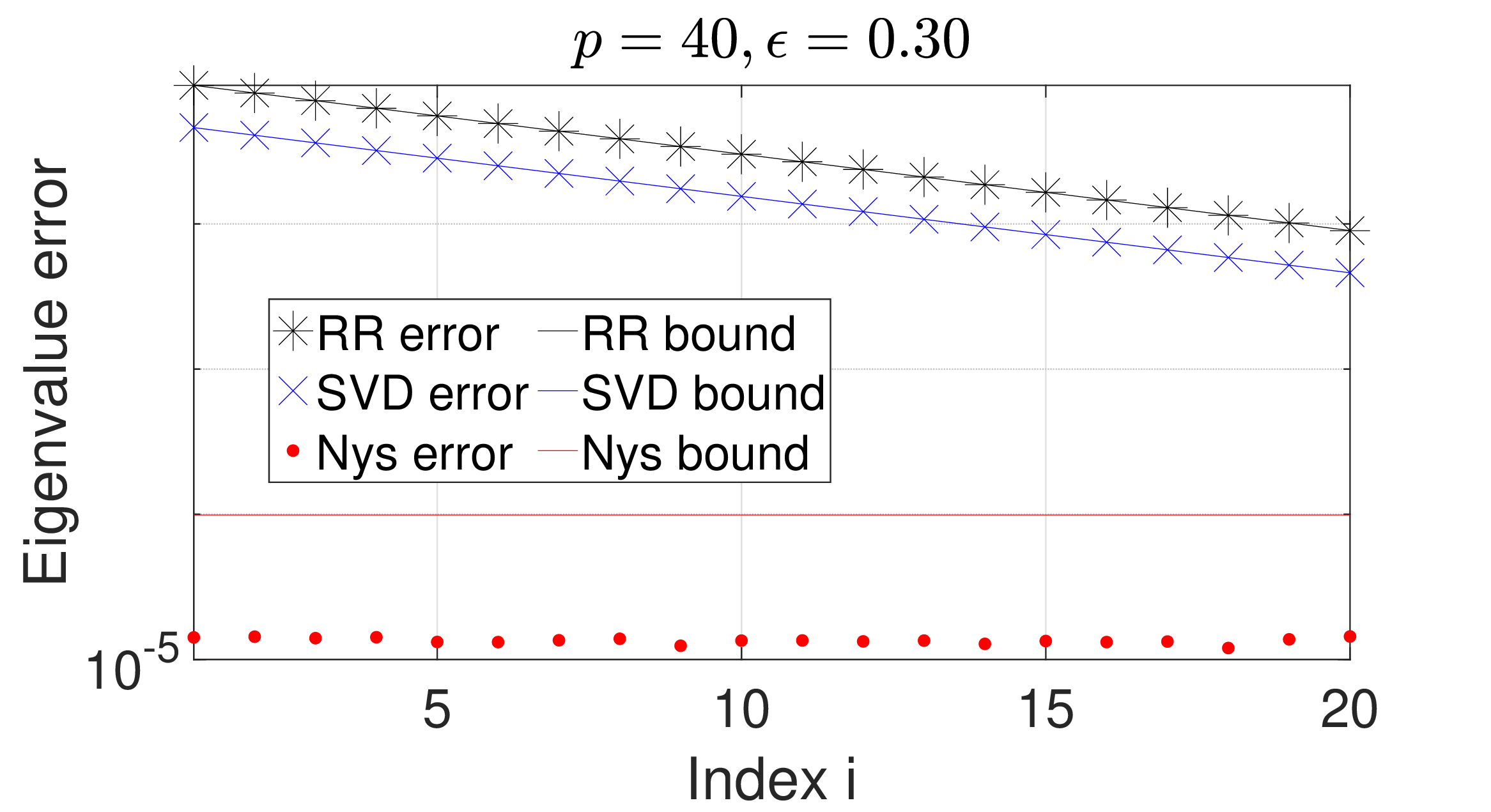}
\caption{Leading eigenvalue errors and error bounds, $n=1000$, $k=20$. The left plots show the sharper \Nystrom\ bounds \eqref{eq:error-nys-sharp} and \eqref{eq:error-bound-nys-general-oversampling}. The right plots show the \Nystrom\ bounds \eqref{eq:error-nys} and \eqref{eq:error-bound-nys-sharp-oversampling}.}
\label{fig:eigenvalue-error-and-bound}
\end{figure}

\subsection{Eigenvector approximation}
We compare \Nystrom, $\mathrm{SVD}_Q$, $\mathrm{SVD}_{AQ}$ and RR individual eigenvector approximation error using $\sin\angle(u_i,\widehat u_i)$ and the same settings as in Figure~\ref{fig:eigenvalue-error-ordering}. Figure~\ref{fig:eigenvector-error} illustrates that extracting from $\image(AQ)$ improves leading eigenvector accuracy for the geometric spectrum. For the algebraic spectrum, none of the methods provides accurate eigenvector approximations. For trailing targets, the shift trick improves \Nystrom, but not the individual eigenvector ill-conditioning issue.
\begin{figure}[htbp]
\centering
\includegraphics[scale=0.28]{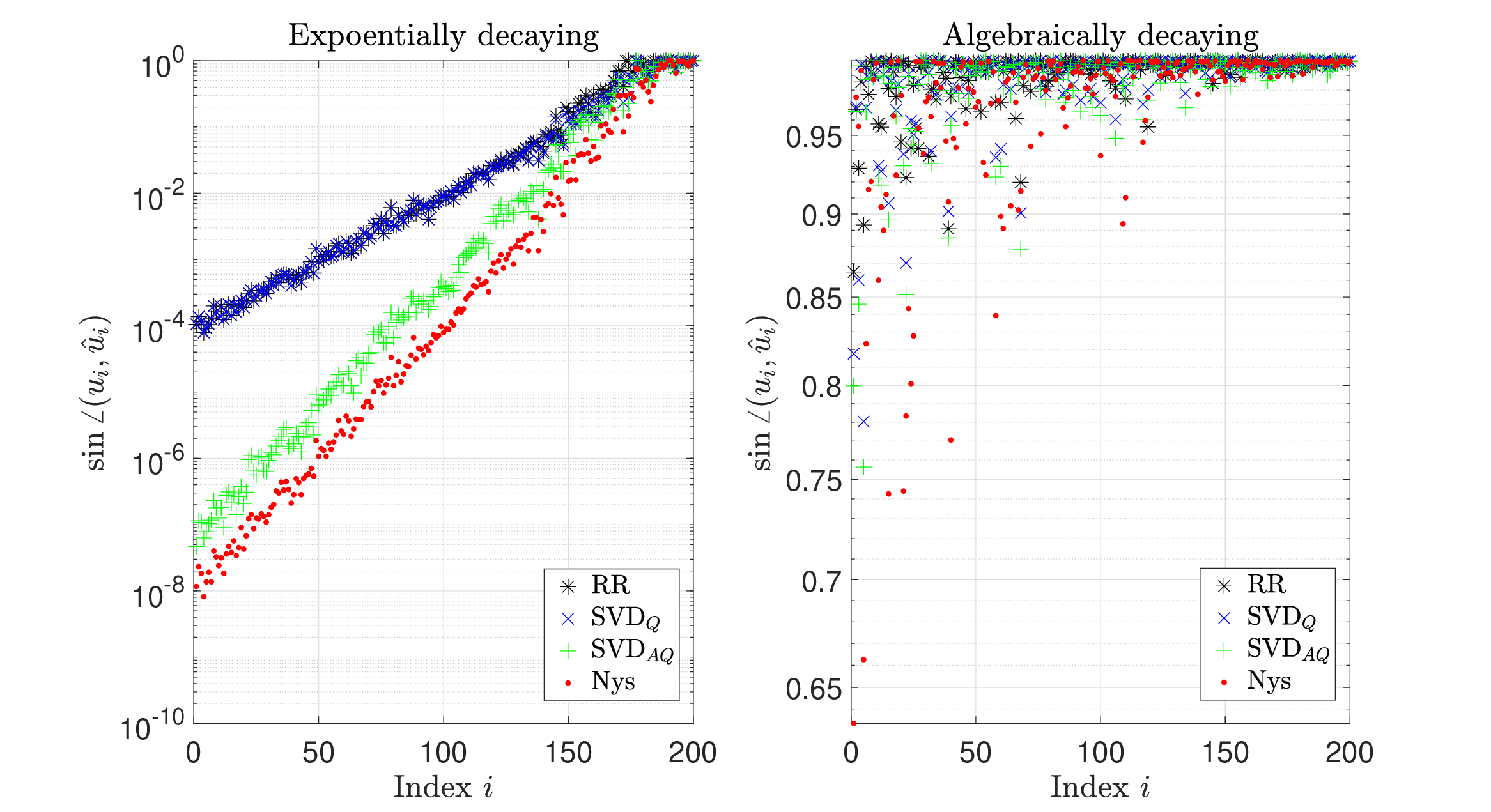}\\[2pt]
\includegraphics[scale=0.28]{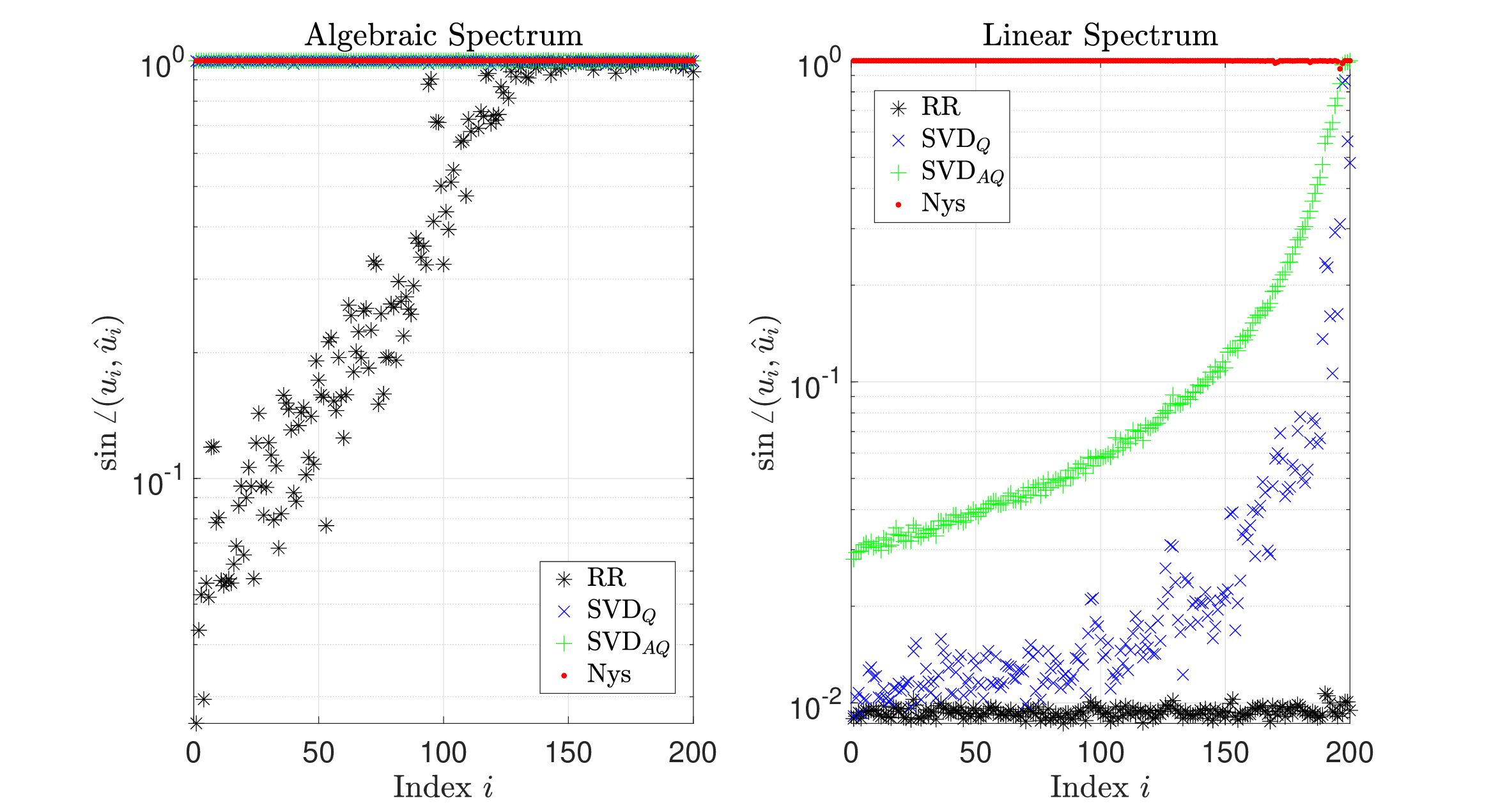}\\[2pt]
\includegraphics[scale=0.28]{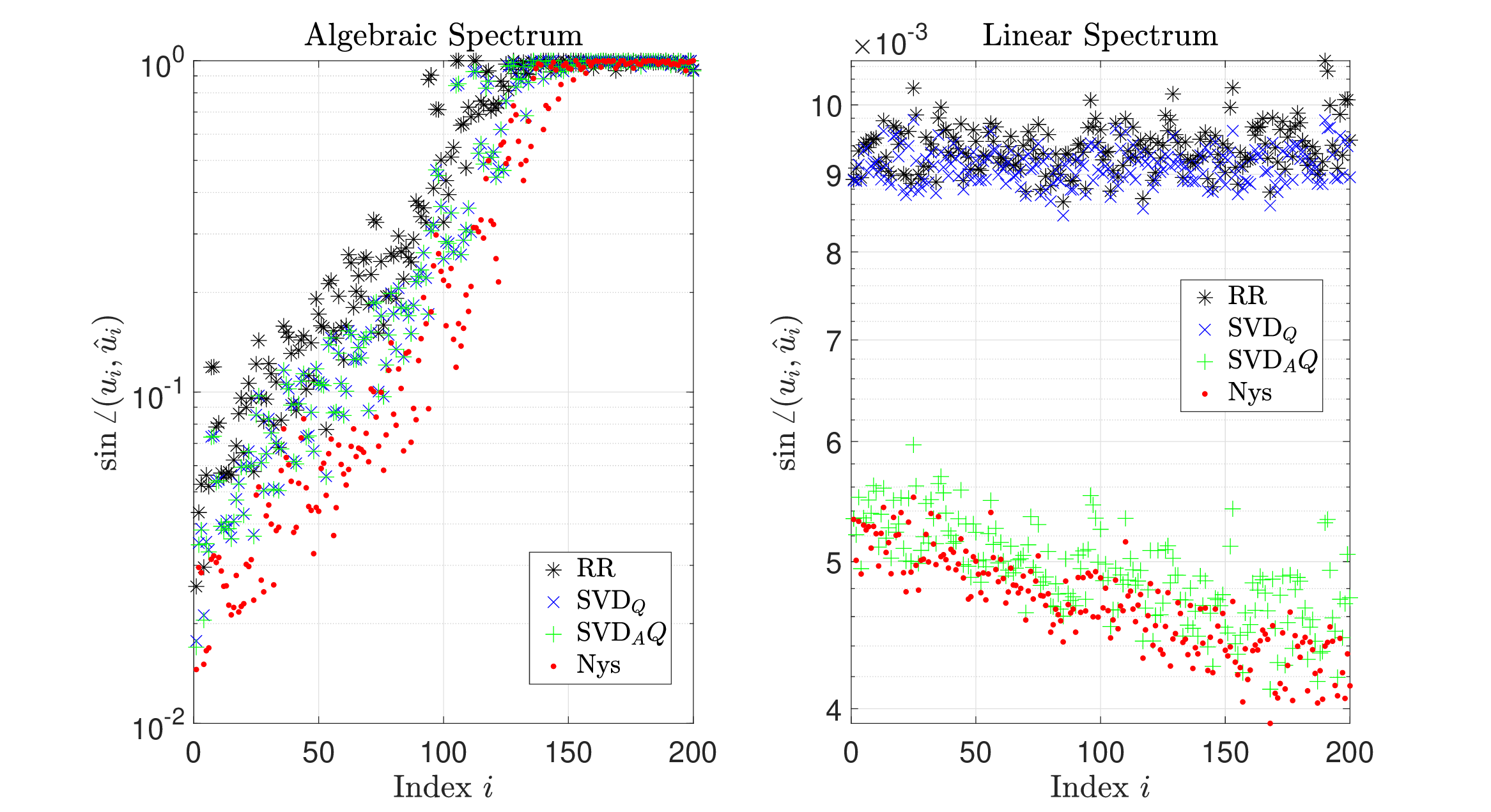}
\caption{Individual eigenvector errors, $k=200$. Top: leading targets for geometric (left) and algebraic (right) spectra. Middle and bottom: trailing targets for linear and algebraic spectra, direct (left) and shifting (right). In shifted plots, $\mathrm{SVD}_{AQ}$ denotes the corresponding left singular vector of $(I-A)Q$.}
\label{fig:eigenvector-error}
\end{figure}

\section{Conclusion and future work}
We have shown how to exploit the PSD property in the extraction step of the subspace methods to improve the accuracy of leading eigenvalue approximations beyond RR via \svdmethod\ and especially \Nystrom. Given a $k$-dimensional subspace, the \Nystrom\ extracted eigenvalues are $O(\epsilon^2 \lambda_{k+1})$ accurate, whereas those of RR and \svdmethod\ are $O(\epsilon^2 \lambda_i)$. As the sine of the largest principal angle between the given subspace and the target leading $k$-dimensional invariant subspace of $A$ close to zero, the accuracy of \Nystrom\ and \svdmethod\ are approximately $\lambda_i(A) / \lambda_{k+1}(A)$ and $2$ times better than RR, respectively. Further, similar theoretical results are applied to the case where we extract $k$ leading eigenvalues from a larger-dimensional subspace, says, a prescribed $p$-dimensional subspace with $p \geq k$. In this case, the accuracy of \Nystrom\ is $O(\epsilon^2 \lambda_{k+1})$ rather than $O(\epsilon^2 \lambda_{p+1})$, and the asymptotic behaviors of the error improvement ratios are the same as the $k$-dimensional subspace case. The situation is the opposite for trailing eigenvalue approximation. However, we can fix this by working with the matrix $-A + \gamma I$ with prescribed $\gamma \geq \lambda_1(A)$. 

Several directions remain for future investigation. Firstly, the current analysis focuses on exterior eigenvalues. A natural next step is to investigate how additional matrix structure, including but not limited to positive semi-definiteness, can be exploited to improve the computation of interior eigenvalues. Secondly, it would be interesting to study whether analogous accuracy improvements can be obtained for generalized Rayleigh-Ritz approximations in generalized eigenvalue problems. Third, it would be interesting to investigate whether analogous improvements in eigenvalue extraction accuracy can be obtained for tensor eigenvalue problems.

\section*{Acknowledgments}
We would like to thank the referees for their helpful comments and suggestions. We acknowledge the use of ChatGPT to assist with grammar correction and proofreading of the manuscript. 

\bibliographystyle{siamplain}
\bibliography{references}
\end{document}